\pdfoutput=1
\documentclass[11pt,english]{article} 

\usepackage[latin9]{inputenc}
\usepackage{babel}
\usepackage[T1]{fontenc}
\usepackage[a4paper]{geometry}
\usepackage{amsthm}
\usepackage{amsmath}
\usepackage{caption}
\usepackage{amssymb}
\usepackage{graphicx}
\usepackage{setspace}
\usepackage{esint}\usepackage{mathbbol}
\usepackage{amssymb}             % AMS Math
\usepackage{subfigure}
\usepackage{epstopdf}
\usepackage{color}
\usepackage{float}
\usepackage{mathbbol}
\usepackage{bm}

\def\R{{\mathbb R}}% real numbers
\def\N{{\mathbb N}}% nonnegative integers
\def\bz{\mathbf z}
\def\bx{\mathbf x}
\def\bk{\mathbf k}

\def\le{\leqslant}% lessoreqal
\newcommand{\re}{\mathrm{Re}}
\newcommand{\im}{\mathrm{Im}}
\newcommand{\eps}{\e}
\newcommand{\e}{\varepsilon}

\newcommand{\pdf}{{\mathcal P}}
\newcommand{\pad}[2]{\frac{\partial{#1}}{\partial{#2}}}

\newtheorem{theorem}{THEOREM}[section]

\newtheorem{lemma}{LEMMA}[section]

\title{Gaussian wave packet transform based numerical scheme for the semi-classical Schr\"odinger equation with random inputs
 \footnote{
S. Jin was supported by NSF grants DMS-1522184 and DMS-1107291: RNMS KI-Net, and NSFC grants No.\,31571071 and No.\,11871297.
L. Liu was partially supported by the funding DOE--Simulation Center for Runaway Electron Avoidance and Mitigation, project No.\,DE-SC0016283. 
G. Russo was partially funded by the ITN-ETN Marie-Curie Horizon 2020 program
ModCompShock, Modeling and computation of shocks and interfaces, project ID: 642768 and NSF grant No.\,DMS-1115252. 
Z. Zhou was partially supported by a start-up fund from Peking University and NSFC grant No.\,11801016. }}
\author{Shi Jin\footnote{School of Mathematical Sciences, Institute of Natural Sciences, MOE-LSC and SHL-MAC, Shanghai Jiao Tong University, Shanghai, China (shijin-m@sjtu.edu.cn)}, Liu Liu\footnote{The Institute for Computational Engineering and Sciences (ICES), University of Texas at Austin, Austin, TX 78712, USA (lliu@ices.utexas.edu)}, Giovanni Russo\footnote{Department of Mathematics and Computer Science, University of Catania, Via A.Doria 6, 95125, Catania, Italy (russo@dmi.unict.it)}\, and Zhennan Zhou\footnote{Beijing International Center for Mathematical Research, Peking University, Beijing, China (zhennan@bicmr.pku.edu.cn)}}

\begin{document}
\maketitle

\begin{abstract}
In this work, we study the semi-classical limit of the Schr\"odinger equation with random inputs, and show that the semi-classical Schr\"odinger equation produces $O(\e)$ oscillations in the random variable space. With the Gaussian wave packet transform, the original Schr\"odinger  equation is mapped to an {\color{black}ordinary differential equation (ODE)} system for the wave packet parameters coupled with a {\color{black}partial differential equation (PDE)} for the quantity $w$ in rescaled variables. Further, we show that the $w$ equation does not produce $\e$ dependent oscillations, and thus it is more amenable for numerical simulations. We propose multi-level sampling strategy in implementing the Gaussian wave packet transform, where in the most costly part, i.e. simulating the $w$ equation, it is sufficient to use $\e$ independent samples. We also provide extensive numerical tests as well as meaningful numerical experiments to justify the properties of the numerical algorithm, and hopefully shed light on possible future directions.
\end{abstract}

\section{Introduction}

In simulating physical systems, which are often modeled by differential equations, there are inevitably modeling errors, imprecise measurements of the initial data or the background coefficients, which may bring about uncertainties to the equation. There has been a growing interest in analyzing such models to understand the impact of these uncertainties, and thus design efficient numerical methods. 

When it comes to quantum dynamics, quantifying the effect of uncertainty is even a trickier task. The solution to the Schr\"odinger equation is a complex valued wave function, whose nonlinear transforms (e.g. position density, flux density) lead to probabilistic measures of the physical observables. Thus, the uncertainty in the Schr\"odinger equation may or may not result in changes in measurable quantities from the quantum state.

We consider the following semi-classical Schr\"odinger equation with random inputs 
\begin{equation} 
i\e\partial_{t}\psi^{\e}(t,\mathbf x,\mathbf z)=-\frac{\e^2}{2}\Delta_{\mathbf x}\psi^{\e}(t,\mathbf x,\mathbf z)+V(\mathbf{x},\mathbf z)\psi^{\e}(t,\mathbf x,\mathbf z),
\label{eq:main equationz}
\end{equation}
\begin{equation}
\label{eq:main equationz0}
\psi^{\e}(0,\mathbf x,\mathbf z)= \psi^{\e}_{\text{in}}(\mathbf x,\mathbf z).
\end{equation}
Here, $\e \ll 1$ is the semi-classical parameter, which is reminiscent of the scaled Plank constant, and $V(\mathbf{x},\mathbf z)$ is the scalar potential function, which is slow-varying and often used to model the external field. The initial condition $\psi^{\e}_{\text{in}}(\mathbf x,\mathbf z)$ will be assumed to be in the form of a semi-classical wave packet, which is a Gaussian wave packet parameterized by the wave packet position, the wave packet momentum, etc. 

The uncertainty is described by the random variable $\bz$, which lies in the random space $I_{\bz}$ with a probability measure $\pi(\bz)d \bz$.  
We introduce the notation for the expected value of $f(\bz)$ in the random variable $\mathbf z$, 
\begin{equation}
\langle f \rangle_{\pi(\mathbf z)} = \int f(\bz) \pi(\mathbf z) d \mathbf z. 
\end{equation}
In this paper, we only consider the uncertainty coming from initial data and potential functions, that is the uncertainty is {\it classical}. 
For example, the external classical field, the wave packet position or the wave packet momentum is uncertain, 
which reflects on the uncertainty in physical observables.

We do not, however, aim to analyze different types of uncertainties in quantum dynamics, but rather, we study how the uncertainty propagates in the semi-classical Schr\"odinger equation. When $\e \ll 1$, it is well known that the Schr\"odinger equation is in the high frequency regime, where the solution generates $O(\e)$ scaled oscillations in space and time. As we shall show in this paper, the solution generically propagates $O(\e)$ scaled oscillations in the $\bz$ variable as well even if the random variable $\bz$ obeys an $\e$ independent probability distribution.  The high frequency of the solution in space, time and uncertainty leads to unaffordable computational cost, which makes conventional numerical methods infeasible. Thus, it is of great interest to design efficient numerical method based on the multiscale nature of the analytical solutions.

In the semi-classical regime, due to the $O(\varepsilon)$ scaled oscillation in the solution to the Schr\"odinger equation, the wave function $\psi^{\varepsilon}$ does
not converge in the strong sense as $\varepsilon\rightarrow0$. The high frequency nature of the wave function of the semi-classical
 Schr\"odinger equation also causes significant computation burdens. If one aims for direct simulation of the wave function, one of the best choices is the time splitting spectral method, as analyzed in \cite{BaoJin} by Bao, Jin and Markowich. See also \cite{reviewsemiclassical, SL-TS, NUFFT, S-LLG}, where the meshing strategy $\Delta t=O(\varepsilon)$
and $\Delta x=O(\varepsilon)$ is sufficient for moderate values of $\varepsilon$. 
Another advantage of the time splitting methods is that if one is only interested in the physical observables, the time step size can be relaxed to $o(1)$, in other words, independently of $\varepsilon$, whereas one still needs to resolve the spatial
oscillations. As we shall show in the paper, when the uncertainty is present, the wave function is also highly oscillatory in the $\bz$ variable, which means the size of samples in random variable grows as $\e \rightarrow 0$ in order to obtain accurate approximations of the quantum dynamics. 

There are quite a few approximate methods
other than directly simulating the semi-classical Schr\"odinger equation, which are valid in the limit $\varepsilon\to 0$, 
such as the level set method and the moment closure method based on the WKB analysis and the Wigner transform, see, for example, \cite{reviewsemiclassical} for a general discussion. In the past few years, many wave packets based methods have been introduced, which reduce the full quantum dynamics to Gaussian wave packets dynamics \cite{Heller,Heller2,Hagedorn}, and thus gain significant savings in computation cost, such as the Gaussian beam method \cite{Ralston,EGB,QianYing}, the Hagedorn wave packet approach \cite{LubichHagedorn,ZhouHagedorn} and the Frozen Gaussian beam method \cite{Kay1,Kay2,FGALuYang,FGALuZhou1,FGALuZhou2}. When random inputs are considered, in theory, one can design numerical methods based on those approximation tools with $\e$ independent samples in $\bz$, however, the approximation errors persist in spite of other potential challenges.

A related work to our current subject is \cite{NGO-17}, where the authors developed the generalized polynomial chaos (gPC)-based 
stochastic Galerkin method for a class of highly oscillatory transport equations containing uncertainties
that arise in semi-classical modeling of non-adiabatic quantum dynamics. 
Built upon and modified from the nonlinear geometrical optics based method, 
this scheme can capture oscillations with frequency-independent time step, mesh size as well as degree of the polynomial.

Clearly, an exact reformulation of the semi-classical Schr\"odinger equation that separates the multiscales in dynamics is desired for efficient simulation. Very recently, Russo and Smereka proposed a new method based on the so-called Gaussian wave packet transform \cite{GWPT,GWPT2,ZR}, which reduces the quantum dynamics to Gaussian wave packet dynamics together with the time evolution of a rescaled quantity $w$, which satisfies another equation of the Schr\"odinger type, in which the modified potential becomes time dependent. We emphasize that the Gaussian wave packet transform is an equivalent reformulation of the full quantum dynamics, and there are no more $\e$ dependent oscillations in the $w$ equation. This motivates us to investigate whether this transform facilitates design of efficient numerical methods for semi-classical Schr\"odinger equations with random inputs.

In this work, we study the semi-classical limit of the Schr\"odinger equation with random inputs, which is the Liouville equation with a random force field, and show that the semi-classical Schr\"odinger equation produces $O(\e)$ oscillations in the $\bz$ variable in general. However, with the Gaussian wave packet transform, the original Schr\"odinger equation is mapped to an ODE system for the wave packet parameters coupled with a PDE for the quantity $w$ in rescaled variables, where the ODE system and the $w$ equation also depend on the random variable. Further, we show that the $w$ equation does not produce $\e$ dependent oscillations in the rescaled spatial variable, thus it is more amenable for numerical simulations. We propose multi-level sampling strategy in implementing the Gaussian wave packet transform, where in the most costly part, simulating the $w$ equation, it is sufficient to use $\e$ independent samples, 
thus the complexity of the whole algorithm has satisfactory scaling behavior as $\e$ goes to $0$. We also provide extensive numerical tests as well as meaningful numerical experiments to justify the properties of the numerical algorithm as well as hopefully shed light on possible future directions.

The rest of the paper is outlined as follows. Section \ref{sec:2} discusses the semi-classical limit of the Schr\"odinger equation
and analyzes the regularity of $\psi$ in the random space. In Section \ref{sec:3}, we introduce the Gaussian wave packet transform 
and prove that the $w$ equation is not oscillatory in the random space. {\color{black}Section \ref{sec:3.5} briefly discusses the comparison between quantum and classical systems.}
Section \ref{sec:4} shows extensive numerical tests by using the 
stochastic collocation method to demonstrate the efficiency and accuracy of our proposed scheme. Relations of different numbers of the collocation points needed in each step of the implementation will be explained and studied 
numerically. Conclusion and future work are given in Section \ref{sec:5}. 

%-----------------------------------------------------------------------------------------------
\section{The semi-classical Schr\"odinger equation with random inputs}
\label{sec:2}

\subsection{The semi-classical limit with random inputs}
\label{Semi-Limit}
In this part, we investigate the semi-classical limit of the  Schr\"odinger equation with random inputs by the Wigner transform 
\cite{Wigner,reviewsemiclassical,Bal,Lions}. Obviously, the potential function $V(\mathbf x,\bz)$ can be decomposed as
\begin{equation}\label{potde}
V(\mathbf x,\bz)=\bar V(\mathbf x)+ N(\mathbf x, \bz),
\end{equation}
such that
\[
\langle V \rangle_{\pi} = \bar V, \quad  \langle N \rangle_\pi =0.
\]

For $f,\, g\in L^{2}(\mathbb{R}^{d})$,
the Wigner transform is defined as a phase-space function
\begin{equation}
W^{\varepsilon}(f,g)\left(t,\bx,\bm \xi\right) = \frac{1}{(2\pi)^{d}}\int_{\R^{d}}e^{i \mathbf y\cdot \bm \xi}\bar{f}
\left(\bx+\frac{\varepsilon}{2}\mathbf y\right)g\left(\bx-\frac{\varepsilon}{2}\mathbf y\right)d \mathbf y, 
\end{equation}
{\color{black}where $\bar{f}$ represents the  complex conjugate of $f$. }

Recall that $\psi^{\varepsilon}(t,\bx)$ is the exact solution of equation
\eqref{eq:main equationz}. Denote $W^{\varepsilon}( t, \bx, \bm \xi, \bz)=W^{\varepsilon}(\psi^{\varepsilon},\psi^{\varepsilon}),$
it is possible to prove that $W^{\varepsilon}$ satisfies the Wigner equation \cite{Wigner, BaoJin}
\begin{equation}
\partial_{t}W^{\varepsilon}+\bm \xi \cdot\nabla_{\bx}W^{\varepsilon}+\Theta[V]W^{\varepsilon}=0,
\end{equation}
in which $\Theta[V]W^{\varepsilon}$ is a pseudo-differential operator acting on $W^\e$  defined by
\begin{equation}
\Theta[V]W^{\varepsilon}:=\frac{i}{(2\pi)^{d}\varepsilon}\int_{\R^{d}}\left(V\left(\bx+\frac{\varepsilon}{2} \bm \alpha\right)-V\left(\bx-\frac{\varepsilon}{2} \bm \alpha\right)\right)\hat{W}^{\varepsilon}(t,\bx,\bm \alpha, \bz)e^{i\bm \alpha\cdot\bm \xi}\,d \bm \alpha, 
\end{equation}
{\color{black}where $\hat W$ represents the Fourier transform of $W$ with respect to the momentum component $\xi$.}

By Weyl's Calculus (see \cite{Wigner}) as $\varepsilon\rightarrow0$, the Wigner measure
$W^{0}=\lim_{\varepsilon\rightarrow0}W^{\varepsilon}({\color{black}\psi^{\varepsilon}}, {\color{black}\psi^{\varepsilon}})$
satisfies the classical Liouville equation 
\begin{equation}
W_{t}^{0}+\bm \xi \cdot\nabla_{\bx}W^{0}-\nabla_{\bx}V \cdot\nabla_{\bm \xi}W^{0}=0,
\end{equation}
with
\begin{equation}
W^{0}(t=0,\bx,\bm\xi,\bz)=W_{I}^{0}(\bx,\bm\xi,\bz):=\lim_{\varepsilon\rightarrow0}W^{\varepsilon}(\psi_{0}^{\varepsilon},\psi_{0}^{\varepsilon}).
\end{equation}
All the limits above are defined in an appropriate weak sense (see \cite{Wigner,Lions}).

With the decomposition of the potential as in \eqref{potde}, the classical Liouville equation  becomes
\begin{equation} \label{eq:WigMea}
W_{t}^{0}+\bm \xi \cdot\nabla_{\bx}W^{0}-\nabla_{\bx} \bar V \cdot\nabla_{\bm \xi}W^{0}-\nabla_{\bx} N \cdot\nabla_{\bm \xi}W^{0}=0.
\end{equation}
This clearly shows, due to the random potential, the bi-characteristics of the Liouville equation contains 
the random force term with $\langle -\nabla_{\bx} N \rangle_\pi =-\nabla_{\bx} \langle  N \rangle_\pi =0$, and the characteristic equations are
\begin{equation}
\left\{
\begin{split}
&\dot \bx =\bm \xi, \\
&\dot {\bm \xi} = -\nabla_{\bx} \bar V-\nabla_{\bx} N.
\end{split}
\right.
\end{equation}
Also, by definition, it is easy to check that $\forall \mathbf z \in I_{\mathbf z}$, $W^\eps$ is real-valued. To sum up, in the semi-classical limit, the Wigner measure $W^0$ picks up the dependence of the random variable $\mathbf z$ though the initial condition and the vector field $\nabla_{\bx} N$.

Next, we discuss if one can derive the averaged equation to integrate out the random variable $\mathbf z$. 
In equation \eqref{eq:WigMea}, by taking the average with respect to the random variable $\bz$, one gets 
\[
\partial_t \langle W^0 \rangle_\pi +\bm \xi \cdot\nabla_{\bx} \langle W^{0} \rangle_\pi-\nabla_{\bx} \bar V \cdot\nabla_{\bm \xi} \langle W^{0}\rangle_{\pi}- \langle \nabla_{\bx} N \cdot\nabla_{\bm \xi}W^{0} \rangle_\pi.
\]
Notice that, in the last term, both $N$ and $W^0$ depend on $\bz$, thus it cannot be directly written as a term involving $\langle W^0 \rangle_\pi$, 
rather it connects with the covariance of $\nabla_{\bx}N$ and $\nabla_{\bm \xi}W^0$. In fact, 
\begin{align*}
- \langle \nabla_{\bx} N \cdot\nabla_{\bm \xi}W^{0} \rangle_\pi & = - \langle (\nabla_{\bx} N - \nabla_{\bx} \langle N \rangle_\pi) \cdot  ( \nabla_{\bm \xi} W^{0} - \nabla_{\bm \xi} \langle W^0 \rangle_\pi )\rangle_\pi - \nabla_{\bx} \langle N \rangle_\pi \cdot\nabla_{\bm \xi} \langle W^0 \rangle_\pi \\
&= - {\rm Cov}( \nabla_{\bx} N, \nabla_{\bm \xi}W^0), 
\end{align*}
since $\nabla_{\bx} \langle N \rangle_\pi =0$. 

To illustrate the effect of the random potential, we consider the following special case
\[
\bar V= \frac{1}{2}|\bx|^2, \quad N = -\bx \cdot \mathbf g(\bz). 
\] 
In this case, the Liouville equation simplifies to 
\begin{equation}\label{sliou}
W_{t}^{0}+\bm \xi \cdot\nabla_{\bx}W^{0}- \bx \cdot\nabla_{\bm \xi}W^{0}+ \mathbf g(\bz) \cdot\nabla_{\bm \xi}W^{0}=0.
\end{equation}
Then, if one considers the following change of variables
\[
\bx=\tilde \bx + \mathbf g (\bz), \quad \bm \xi = \tilde{\bm \xi},
\]
then, equation \eqref{sliou} becomes
\begin{equation}\label{sliou2}
W_{t}^{0}+ \tilde{\bm \xi} \cdot\nabla_{\tilde \bx}W^{0}- \tilde{\bx} \cdot\nabla_{\tilde{\bm \xi}}W^{0}=0.
\end{equation}
And correspondingly, the initial condition becomes
\[
W^{0}(t=0,\tilde \bx, \tilde{\bm\xi},\bz)=W_{I}^{0}(\tilde \bx + \mathbf g (\bz), \tilde{\bm \xi}).
\]
Thus, the average with respect to the random variable $\bz$ can be taken and one gets a closed equation for $\langle W^{0} \rangle_\pi$,
\begin{equation}\label{sliou3}
\partial_t \langle W^{0} \rangle_\pi+ \tilde{\bm \xi} \cdot\nabla_{\tilde \bx} \langle W^{0} \rangle_\pi - \tilde{\bx} \cdot\nabla_{\tilde{\bm \xi}}\langle W^{0} \rangle_\pi=0.
\end{equation}
and 
\[ \langle W^{0} \rangle_\pi (t=0,\tilde \bx, \tilde{\bm\xi})= \langle W_{I}^{0}(\tilde \bx + \mathbf g (\cdot) ,\tilde{\bm \xi}) \rangle_\pi.\]

This example shows that the randomness in the slow-varying potential changes the transport part of the Liouville equation, although in the averaged equation the transport structure may be even unchanged. 

When the potential is random and with a general form of $z$ dependence, the two processes of  a) first pushing $\e\to 0$ then taking expected value in $z$ of the classical limit (the Liouville equation); 
and b) first taking expected value in $z$ on the Wigner equation then letting $\e\to 0$ do not commute. Though numerically, our simulation results seem to suggest the commutation of these two iterated processes. In later section we will make a numerical comparison to show that the expected values of the position density and flux do converge in the $\varepsilon\to 0$ limit. 

To conclude this part, we remark that in \cite{Bal,AMS98} and subsequent works, the authors have considered apparently a related but fundamentally different model, where the unperturbed system is the semi-classical Schr\"odinger with fast-varying smaller magnitude
and random perturbation in the potential is also fast-varying. 
And they show that the randomness in that scaling introduces additional scattering terms in the limit equations, while in our case the randomness only persist in the initial data and the force field in the limit equation.

%--------------------------------------------------------------------------
\subsection{Regularity of $\psi$ in the $\bz$ variable}

The semi-classical Schr\"odinger equation is a family of dispersive wave equations parameterized by $\e \ll 1$, and it is well known that the wave equation propagates $O(\e)$ scaled oscillations in space and time. However, it is not clear yet whether the small parameter $\e$ induces oscillations in the random variable $\mathbf z$.

Here and in subsection \ref{Reg-W}, we will conduct a regularity analysis of $\psi$ in the random space, which enables us to study the oscillatory behavior of solutions in the random space, which gives guidance on how many collocation points needed
in each step of the collocation method should depend on the scaled constant $\e$. 

To investigate the regularity of the wave function in the $\bz$ variable, we check the following averaged norm 
\begin{equation}
\label{energy} ||f||_{\Gamma }:= \left (\int_{I_z}\int_{\mathbb R^3}\left| f(t, \mathbf x, \bz)\right|^2\,  d{\mathbf x}\pi(\bz)d{\bz} \right)^{\frac 1 2}. 
\end{equation}
To be more precise, (\ref{energy}) denotes the square root of the expected value in $\bz$ of the square of the $L^2(x)$ norm of $f$. 
We name it $\Gamma$-norm for short.

One first observes that
$\forall\, \mathbf z \in I_{\mathbf z}$, 
\[
	\frac{\partial}{\partial t} \| \psi^{\e}\|^2_{L^2_{\mathbf x}}(t,\bz)=0, 
\]
thus 
\[
\frac{d}{dt} \| \psi^{\e}\|_{\Gamma}^2=0,
\]
which means the $\Gamma$-norm of the wave function $\psi^\e$ is conserved in time,
\[
\| \psi^{\e}\|_{\Gamma} (t) =  \| \psi^{\e}_{\text{in}}\|_{\Gamma}\,. 
\]

However, we show in the following that $\psi^\e$ has $\e$-scaled oscillations in $\mathbf z$ even if $V$ and $\psi_{\text{in}}^{\e}$ do not have 
$\e$-dependent oscillations in $\mathbf z$. 
We first examine the first-order partial derivative of $\psi^\e$ in $z_1$, and denote $\psi^1= \psi^\e_{z_1}$ and $V^1=V_{z_1}$, then by differentiating the semi-classical schr\"odinger equation (\ref{eq:main equationz}) with respect to $z_1$, one gets
\[
    i \e \psi^1_t =- \frac{\e^2}{2}\Delta_{\mathbf x} \psi^1 + V^1 \psi^\e + V \psi^1.
\]

By direct calculation, 
\begin{align*}
\frac{d}{dt} \| \psi^{1}\|^2_{\Gamma} &= \int \bigl (\psi^1_t \bar \psi^1 + \psi^1 \bar \psi^1_t \bigr) \pi d \mathbf x d \mathbf z \\
& = \int \bigl(\frac{1}{i\e} V^1 \psi^\e \bar \psi^1  - \frac{1}{i\e} V^1 \psi^1 \bar \psi^\e \bigr) \pi d \mathbf x d \mathbf z \\
& \le \frac{2}{\e} \| \psi^{1}\|_{\Gamma}\, \| V^1 \psi^{\e}\|_{\Gamma}\,,
\end{align*}
where the Cauchy-Schwarz inequality and the Jensen inequality are used in the last step, more specifically, 
\begin{align*}
&\int V^{1}\psi^{\e}\bar\psi^{1} dx \leq \left( \int (V^{1}\psi^{\e})^2 dx \right)^{1/2} \left( \int (\bar\psi^{1})^2 dx\right)^{1/2}, \\
&\int\int V^{1}\psi^{\e}\bar\psi^{1} dx\, \pi(z)dz \leq \left(\int \left(\int V^{1}\psi^{\e}\bar\psi^{1} dx\right)^2 \pi(z)dz\right)^{1/2} \leq ||V^{1}\psi^{\e}||_{\Gamma}\, 
||\psi^{1}||_{\Gamma}\,.
\end{align*}
Thus
\[
    \frac{d}{dt} \| \psi^{1}\|_{\Gamma} \le\frac{1}{\e} \| V^1 \psi^{\e}\|^2_{\Gamma}\,, 
\]
For $t=O(1)$, the pessimistic estimate implies
\[  \| \psi^{1}\|_{\Gamma} =O\bigl(\e^{-1}\bigr). 
\]
Moreover, for $\mathbf k=(k_1,k_2,\cdots,k_n) \in \N^n$, denote $|\mathbf k|=\sum_{j=1}^n k_j $, we can similarly conclude that 
\begin{equation}\label{est:psiz}
 \|\partial_{\mathbf z}^{\mathbf  k}  \psi^{\e}\|_{\Gamma} =O\bigl( {\e^{-|\mathbf k|}}\bigr). 
\end{equation}

Although the estimates above are apparently pessimistic, we would like to show that the high frequency oscillations in $\mathbf z$ can be seen
in the following example. For simplicity, we consider $x$ be one dimensional. If the potential $V$ is quadratic in $x$, it has been shown by Heller in \cite{Heller} that,
\begin{equation}
\label{Heller}
\phi(x,t)=\exp \left[ i \frac{\alpha(t)\bigl(x-q(t)\bigr)^2-p(t)\bigl(x-q(t)\bigr)+ \gamma(t) }{\e}\right]
\end{equation}
is an exact solution to the semi-classical Schr\"odinger equation, provided that, $q(t)$, $p(t)$, $\alpha(t)$ and $\gamma(t)$ satisfy the following system of equations
\begin{equation} \label{psys2}
\left\{ \begin{array}{l}
 \dot q = p,
\\ \dot p =  - V_q(q),
\\ \dot \alpha = -2 \alpha^2 -\frac{1}{2}V_{qq}(q),
\\ \dot \gamma =\frac{1}{2}p^{2}-V(q)+i\e \alpha.
\end{array}\right.
\end{equation}

Due to the same reason,
\[
\Phi(t,x,\mathbf z)=\exp \left[ i \frac{\alpha(t,\mathbf z)\bigl(x-q(t,\mathbf z)\bigr)^2-p(t,\mathbf z)\bigl(x-q(t,\mathbf z)\bigr)+ \gamma(t,\mathbf z) }{\e}\right]
\]
is an exact solution to equation (\ref{eq:main equationz}), when potential $V(x, \mathbf z)$ is quadratic in $x$.  
Clearly, this specific solution saturates the estimate \eqref{est:psiz}, which implies, even if initially $\psi^\e$ is smooth in $\mathbf z$, it will pick up $\e$-dependent oscillations in the $\mathbf z$ variables. 
In Section \ref{sec:4}, we will also show numerically the $\e$-scaled oscillations of such wave functions in the $\mathbf z$ variable. 

To conclude this section, we emphasize the numerical challenges with respect to the random variable $\bz$. Due to the oscillatory behavior in $\mathbf z$,
if one applies the generalized polynomial chaos (gPC)-based stochastic methods
directly to the semi-classical Schr\"odinger equation, one needs at least $\e$-dependent basis functions or quadrature points to get an accurate approximation.
The stochastic collocation method will be discussed in detail in subsection \ref{sec:4a}. 

%--------------------------------------------------------------------
\section{The semi-classical Schr\"odinger equation and the Gaussian wave packet transformation}
\label{sec:3}

To overcome the numerical burdens in sampling the random variable $z$, we introduce the Gaussian wave packet transformation (abbreviated by GWPT), which has been proven to be a very efficient tool for computing the Schr\"odinger equation in the high frequency regime. In essence, the GWPT equivalently transforms the highly oscillatory wave equation to {\color{black}an equation for a new rescaled wave function $w$}, thus facilitate the design of efficient numerical methods. 
{\color{black}At variance with most methods based on the Gaussian beam or Gaussian wave packet, the
GWPT approach is not based on an asymptotic expansion in $\e$, therefore
it is equivalent to the original Schr\"odinger equation for {\it all} $\e$.}
The main goal of this section is to study the {\it uniform regularity} of $w$ in the $\bz$ variable, 
and conclude that the number of basis functions or quadrature points is {\it independent of $\e$},
if applying the SG or SC method in our GWPT framework. 

\subsection{Review of the Gaussian wave packet transformation}

First, we briefly summarize the Gaussian wave packet transformation applied to the semi-classical Schr\"odinger equation with random inputs, 
which is a natural extension of the GWPT method for the deterministic problem \cite{GWPT}. Consider the semi-classical Schr\"odinger equation given by (\ref{eq:main equationz})--(\ref{eq:main equationz0}). Note that in this work, random inputs are assumed to be classical, thus we only consider the cases when the wave packet position and momentum in the GWPT parameters depend on the random variable $\bz$. 

We start by the following {\it ansatz}
\begin{equation}\label{ansatz}
\psi(t,\mathbf x,\bz)= \widetilde w(t,\boldsymbol\xi,\bz)\exp \left(g(t,\boldsymbol\xi,\bz)\right):=\widetilde w(t, \boldsymbol\xi, \bz)\exp \left(i\left(\boldsymbol\xi^{T}
\boldsymbol{\alpha}_R\, \boldsymbol\xi + {\mathbf p}^{T}\boldsymbol\xi +\gamma\right)/\e\right),
\end{equation}
where $\boldsymbol\xi=\mathbf x - \mathbf q$, $\boldsymbol{\alpha}_R$ is a real-valued symmetric matrix and $\gamma$ is a complex-valued scalar.

Denote $\boldsymbol{\alpha}_R=\re (\boldsymbol{\alpha})$, $\boldsymbol{\alpha}_I=\im (\boldsymbol{\alpha})$ and 
$\boldsymbol{\alpha}=\boldsymbol{\alpha}_R + i \boldsymbol{\alpha}_I$. 
Insert the ansatz (\ref{ansatz}) into (\ref{eq:main equationz}), then $\widetilde w(t, \boldsymbol\xi, \bz)$ satisfies
\begin{eqnarray*}
	\widetilde w_{t}  = -2\, \boldsymbol\xi^{T} \boldsymbol{\alpha}_R \nabla_{\boldsymbol\xi} \widetilde w + \frac{i \e}{2} \Delta_{\boldsymbol\xi}\widetilde w
	 - \frac{i}{\e} \left(U_{r}+2\, \boldsymbol\xi^T \boldsymbol{\alpha}_I^2\, \boldsymbol\xi \right)\widetilde w,
\end{eqnarray*}
provided $\mathbf p$, $\mathbf q$, $\boldsymbol\alpha$ satisfy the following equations
\begin{equation} \label{psys}
\left\{ \begin{array}{l}
 \dot {\mathbf q} = \mathbf p,
\\ \dot {\mathbf p} =  -\nabla V(\mathbf q),
\\ \dot {\boldsymbol\alpha} = -2\, \boldsymbol{\alpha}^2 - \frac{1}{2}\, \nabla\nabla V(\mathbf q), 
\\ \dot \gamma = \frac{1}{2}\, {\mathbf p}^T \mathbf p - V(\mathbf q) + i\e\, {\rm Tr}(\boldsymbol{\alpha}_R), 
\end{array}\right.
\end{equation}
while 
\[
	U_r=V(\boldsymbol\xi + \mathbf q)-V(\mathbf q)-\boldsymbol\xi^T\, \nabla V(\mathbf q)-\frac{1}{2}\, \boldsymbol\xi^T\, \nabla^2 V(\mathbf q)\boldsymbol\xi. 
\]
At last, introduce the change of variables $\widetilde w(t,\boldsymbol\xi, \bz)=w(t,\boldsymbol\eta, \bz)$, where
\begin{equation}\label{changeV}
	\boldsymbol\eta={\mathbf B\boldsymbol\xi}/\sqrt{\e}, 
\end{equation}
with
\[
\dot {\mathbf B} = - 2 \mathbf B\, \boldsymbol{\alpha}_R,  \qquad \mathbf B(0)=\sqrt{\boldsymbol{\alpha}_I(0)}\,,
\]
then $\boldsymbol{\alpha}_I = {\mathbf B}^T\, \mathbf B$ and 
\begin{equation}\label{weq}
w_{t}=\frac{i}{2}\, {\rm Tr}\left({\mathbf B}^T\, \nabla_{\boldsymbol\eta}^2 w\, {\mathbf B}\right) - 2 i\, \boldsymbol\eta^{T}(\mathbf {B}^T)^{-1}\,\boldsymbol{\alpha}_I^2\, \mathbf {B}^{-1}\boldsymbol\eta w+\frac{1}{i\e}\, U_{r}w, 
\end{equation}
Note that in the $\boldsymbol\eta$ variable, 
\[
\frac{1}{i\e}U_{r}=O(\sqrt \e),
\]
so the $w$ equation (\ref{weq}) is not oscillatory in $\boldsymbol\eta$ nor in $t$. Furthermore, if one drops those $O(\sqrt{\e})$
terms, one expects to recover the leading order Gaussian beam method \cite{GWPT2}.

In our numerical tests, we will only consider the initial data $\psi$ given by a Gaussian wave packet, i.e.
\begin{equation}\label{initial_GWP}
\psi(0,\mathbf x,\bz)=\exp \left(i\left(\boldsymbol\xi^{T}
\boldsymbol{\alpha}\, \boldsymbol\xi + {\mathbf p}^{T}\boldsymbol\xi +\gamma\right)/\e\right),
\end{equation}
More general initial conditions, (with a numerical support proportional to $\sqrt{\varepsilon}$) can  be approximated with the desired accuracy as a superposition of relatively small number of Gaussian wave packets. See \cite[Section 2.6]{GWPT} for more details. 
{\color{black}Another point of view to shed some light on the GWPT formulation is the following: 
Direct solution of the Schroedinger equation for a modulated wave-packet requires a lot of grid points, in most of which 
the wave function is almost zero, while the change of variable allows to work with a fixed computational domain of length $O(1)$, where 
the non oscillatory transformed wave function $w$ can be resolved with a relatively small number of grid points. }
%---------------------------------------------------------------------------

\subsection{Regularity of $w$ in the $\bz$ variable}
\label{Reg-W}

It is well understood that the $w$ equation no longer propagates $\e$-dependent oscillations in space or in time. We show in the following that the $w$ equation is not oscillatory in the $\bz$ variable either.
In this section, we assume the spatial variable $x,\,\eta$ is one dimensional to simplify the analysis. Now the $w$ equation (\ref{weq}) reduces to 
\begin{equation}
\label{w-eq1}
w_t = \frac{i}{2} \alpha_I  w_{\eta \eta} - 2 i \alpha_I \eta^2 w+ \frac{U_r}{i \e} w, 
\end{equation}
where
\[
U_r (t,\eta,\mathbf z)= V(q+\sqrt{\e}  B^{-1} \eta, \mathbf z ) - V(q,\mathbf z)-\sqrt{\e}B^{-1}\eta\, V_x(q,\mathbf  z)- \frac{1}{2}\bigl (\sqrt{\e}B^{-1}\eta\bigr)^2\, V_{xx}(q,\mathbf z), 
\]
and $B=\sqrt{\alpha_I}$\,. Observe that $w=w(t,\eta;\mathbf z)$ is a function of independent variables $(t,\eta)$, but it obtains the dependence of the parameter $\mathbf z$ through the coefficients. We emphasize that, although the change of variable \eqref{changeV} is $\mathbf z$ dependent, in the $w$ equation, $\eta$ and $\mathbf z$ are independent variables. This is due to the fact that it is $w=w(t,\eta;\mathbf z)$, not 
$w=w(t,\eta(\mathbf z);\mathbf z)$. 
With the random inputs, the Gaussian wave packet transform is straightforward for all $\bz$. 
Besides the time dependence, the Gaussian wave packet parameters $q$, $p$, $\alpha$, $B$ and $\gamma$ also depend on $\mathbf z$. 
The smooth components $w$ and $U_r$ depend on $\bf z$. Then, it is not yet clear whether $w$ has $\e$ dependent $z$ derivatives.

We make the assumption that the potential is infinitely smooth with bounded derivatives in both $x$ and $\mathbf z$, namely, for $m\in \N$ and $\mathbf k=(k_1,k_2,\cdots,k_n) \in \N^n$, there exists a constant $C_{m,\mathbf k}$ such that,
\begin{equation}\label{assump1}
|\partial_x^m \partial_{\mathbf z}^{\mathbf  k}V|\le C_{m,\mathbf k}, 
\end{equation}
where $\partial_{\mathbf z}^{\mathbf  k}=\partial_{z_1}^{k_1}\cdots \partial_{z_n}^{k_n}$. 
We also assume that the $w$ equation (\ref{w-eq1}) is equipped with an initial condition, $w(0,\eta,\mathbf z)=w_{\text{in}}(\eta,\mathbf z)$, which has a $O(1)$ sized support, 
satisfying {the following assumption:} for $m\in \N$ and  $\mathbf k=(k_1,k_2,\cdots,k_n) \in \N^n$, there exists an $\e$-independent constant $C_{m,\mathbf k}$ such that, 
\begin{equation}\label{assump2}
\| \partial_\eta^m\partial_{\mathbf z}^{\mathbf  k} w_{\text{in}}\|_{\Gamma}\le C_{m,\mathbf k}. 
\end{equation}

Our goal is to show that the smoothness in the $\bz$ variable will be preserved in time. 
Before calculating the regularity of $w$ in the $\bz$ variable, we first show Lemma \ref{lemma:parameters}
and Lemma \ref{Newlemma:Ur}, which will be used for the main result of this section namely Theorem \ref{w_z}.

\begin{lemma}
\label{lemma:parameters}
Assume boundedness condition \eqref{assump1}, and initial data for wave packet parameters satisfying the following: there exists an $\eps$ independent constant $C_{\mathbf K}$, such that
\[
|\partial_{\mathbf z}^{\mathbf  k} q(0)| \le C_{\mathbf K}, \quad |\partial_{\mathbf z}^{\mathbf  k} p(0)| \le C_{\mathbf K}, \quad |\partial_{\mathbf z}^{\mathbf  k} \alpha(0) | \le C_{\mathbf K}.
\]
Then,  there exists an $\eps$ independent constant $C_{T,\mathbf K}$, such that for $t\in[0,T]$,
\[
|\partial_{\mathbf z}^{\mathbf  k} q(t)| \le C_{T,\mathbf K}, \quad |\partial_{\mathbf z}^{\mathbf  k} p(t)| \le C_{T,\mathbf K}, \quad |\partial_{\mathbf z}^{\mathbf  k} \alpha(t) | \le C_{T,\mathbf K}.
\]
\end{lemma}

The proof follows standard estimations of the ODE system of the parameters, which we shall omit here. We also remark that we have only listed the parameters needed for showing the regularity properties of the $w$ equation (\ref{w-eq1}), but this argument clearly works for other wave packets parameters as well.
\begin{lemma}
\label{Newlemma:Ur}
With the boundedness assumptions \eqref{assump1}, for all $\mathbf k=(k_1,k_2,\cdots,k_n) \in \N^n$, it is 
\[ 
\partial_{\bz}^{\bk}U_r=O(\e^{\frac 3 2}). 
\]
\end{lemma}
\begin{proof}
Recall that 
\begin{equation}
\label{Ur}
U_r (t,\eta,\mathbf z)= V(q+\sqrt{\e}  B^{-1} \eta, \mathbf z ) - V(q,\mathbf z)-\sqrt{\e}B^{-1}\eta\, V_x(q,\mathbf  z)- \frac{1}{2}\bigl (\sqrt{\e}B^{-1}\eta\bigr)^2\, V_{xx}(q,\mathbf z).
\end{equation}
We observe that, assumption \eqref{assump1} together with the Taylor's Theorem implies, for $m\in \N$ and $\mathbf k=(k_1,k_2,\cdots,k_n) \in \N^n$,
\begin{multline} \label{est:T}
T_{m,\mathbf k}:=\partial_x^m \partial_{\mathbf z}^{\mathbf  k}V(q+\sqrt{\e}  B^{-1} \eta, \mathbf z )-\partial_x^m \partial_{\mathbf z}^{\mathbf  k}V(q, \mathbf z )- \sqrt{\e}  B^{-1} \eta \partial_x^{m+1} \partial_{\mathbf z}^{\mathbf  k}V(q, \mathbf z )\\
- \frac{1}{2} \bigl (\sqrt{\e}B^{-1}\eta  \bigr)^2 \partial_x^{m+2} \partial_{\mathbf z}^{\mathbf  k}V(q, \mathbf z )=O(\e^{\frac 3 2}).
\end{multline}
Thus, it is clear that $|U_r| = O(\e^{\frac 3 2})$.

Next, we examine the first order derivative in $z_1$. By direct calculation, 
\begin{align*}
\partial_{z_1}U_r &= \partial_{z_1}V(q+\sqrt{\e}B^{-1} \eta, \bz)  + \partial_{x}V(q+\sqrt{\e}B^{-1}\eta, \bz)\partial_{z_1}(q+\sqrt{\e}B^{-1}\eta) \\ 
&\quad - (\partial_{z_1}V(q, z) + \partial_{x}V(q, \bz)\partial_{z_1}q) \\
&\quad - \left(\partial_{z_1}( \sqrt{\e} B^{-1}\eta)\partial_x V(q, \bz) +  \sqrt{\e}B^{-1}\eta\, \partial_{x z_1}V(q, \bz)+ \sqrt{\e}B^{-1}\eta\, \partial_{xx}V(q, \bz) \partial_{z_1}q \right) \\ 
&\quad - \frac{1}{2}\e\left( \partial_{z_1}(B^{-1}\eta)^2 \partial_{xx}V(q, \bz) + (B^{-1}\eta)^2 \partial_{xx z_1}V(q, \bz)+ (B^{-1}\eta)^2 \partial_{xx x}V(q, \bz) \partial_{z_1} q\right) \\
& =\partial_{z_1}V(q+\sqrt{\e}B^{-1} \eta, \bz)  - \partial_{z_1}V(q, z) \\
& \quad - \sqrt{\e}B^{-1}\eta\, \partial_{x z_1}V(q, \bz) - \frac 1 2 (B^{-1}\eta)^2 \partial_{xx z_1}V(q, \bz) \\
& \quad + \partial_{x}V(q+\sqrt{\e}B^{-1}\eta, \bz)\partial_{z_1}q -\partial_{x}V(q, \bz)\partial_{z_1}q \\
& \quad - \sqrt{\e}B^{-1}\eta\, \partial_{xx}V(q, \bz) \partial_{z_1}q - \frac 1 2  (B^{-1}\eta)^2 \partial_{xx x}V(q, \bz) \partial_{z_1} q \\
& \quad +  \left( \partial_{x}V(q+\sqrt{\e}B^{-1}\eta, \bz)-  \partial_x V(q, \bz) - \sqrt{\e}B^{-1}\eta\, \partial_{xx}V(q, \bz)\right) \partial_{z_1}(\sqrt{\e}B^{-1}\eta). 
\end{align*}
Thus, by \eqref{est:T} and Lemma \ref{lemma:parameters}, we conclude that 
$\partial_{z_1}U_r =O(\e^{\frac 3 2})$. 

By induction, one easily sees that $\partial_{\mathbf z}^{\mathbf  k} U_r $ is a summation of products of $T_{m,\mathbf k}$
(add up with some other terms). Hence, by \eqref{est:T} and Lemma \ref{lemma:parameters}, it can be concluded that $\partial_{\mathbf z}^{\mathbf  k} U_r =O(\e^{\frac 3 2})$, 
and the lemma follows.
\end{proof}

We give the definition of the averaged norm of $w=w(t, \eta, \bz)$, 
$$ ||w||_{\mathcal T}^2:=  \langle \| w \|_{L^2(\eta)}^2 \rangle_{\pi(\bz)},  $$
which is analogous to the $\Gamma$--norm defined in (\ref{energy}) for $\psi$ while using $\eta$ variable in the $L^2(\eta)$ norm here.
We now present the main theorem of this section:

\begin{theorem}\label{w_z}
With the boundedness assumptions \eqref{assump1} and conditions on the initial data 
\eqref{assump2}, the $w$ equation (\ref{w-eq1}) preserves the regularity in the following sense: for a fixed $T>0$, $\mathbf k=(k_1,k_2,\cdots,k_n) \in \N^n$,  there exists an $\e$-independent constant  $M_{T,\mathbf k}$, such that for $0\le t \le T$,
\begin{equation}\label{reg_wz}
 || \partial_{\bz}^{\bk}w||_{\mathcal T}\le M_{T,\mathbf k}\,. 
\end{equation}
\end{theorem}
\begin{proof}
The $w$ equation can be written as 
\begin{equation}\label{w_t} w_t = \frac{i}{2}\alpha_{I}w_{\eta\eta}- i\tilde U w, \end{equation}
where $$\tilde U = 2\alpha_{I}\eta^2 +\frac{U_r}{\e}, $$
and $U_r$ given in (\ref{Ur}). 
We first look at the $\mathcal T$-norm of $w$, 
\begin{align*}
         & \frac{d}{d t} ||w||_{\mathcal T}^2  = \int (w \bar w)_t\, d\eta \pi(\bz)d\bz = \int (w_t \bar w+ w \bar w_t)\, d\eta\pi(\bz)d\bz \\
         &\qquad\quad = \int \bigl(\frac{i}{2}\alpha_{I}w_{\eta\eta} - i\tilde U w)\bar w + w(-\frac{i}{2}\alpha_{I}\bar w_{\eta\eta} + 
          i \tilde U \bar w)\bigr)\, d\eta\pi(\bz)d\bz =0, 
\end{align*}
which implies the averaged norm of $w$ is preserved, $||w||_{\mathcal T}= ||w_{\text{in}}||_{\mathcal T}$. 

Differentiating (\ref{w_t}) with respect to $z_1$, by the chain rule, one has
$$\partial_t \partial_{z_1}w = \frac{i}{2}\bigl(\partial_{z_1}\alpha_{I}\, w_{\eta\eta} 
+ \alpha_{I}   w_{\eta\eta z_1}\bigr) 
- i \partial_{z_1}\tilde U w - i \tilde U  \partial_{z_1}w. $$
By direct calculation (omit the $\mathcal T$-subscript in the norm $||\cdot||_{\mathcal T}$ for notation simplicity), 
\begin{align*}
\frac{d}{dt} ||\partial_{z_1}w||^2
& = \int\left(\partial_t\partial_{z_1}w\, \partial_{z_1}\bar w + \partial_{z_1}w\, \partial_t\partial_{z_1}\bar w\right) d\eta\pi(\bz)d\bz \\ 
& = \int \left( \left[\frac{i}{2}\left(\partial_{z_1}\alpha_{I}\, w_{\eta\eta} + \alpha_{I}  w_{\eta\eta z_1}\right)
- i \partial_{z_1}\tilde U w - i\tilde U  \partial_{z_1}w\right] \partial_{z_1}\bar w \right.  \\
&\quad  \left. + \partial_{z_1}w \left[-\frac{i}{2}\bigl(\partial_{z_1}\alpha_{I}\, \bar w_{\eta\eta}+ \alpha_{I} \bar w_{\eta\eta z_1}\bigr) + i\partial_{z_1}\tilde U \bar w + i\tilde U \partial_{z_1}\bar w\bigr)\right]\right) d\eta\pi(\bz)d\bz \\
& = \int \left[\frac{i}{2}\partial_{z_1}\alpha_{I}(w_{\eta\eta}\, \partial_{z_1}\bar w - \partial_{z_1}w\, \bar w_{\eta\eta})  + i \partial_{z_1}\tilde U (-w\, \partial_{z_1}\bar w + \partial_{z_1}w\, \bar w)\right] d\eta\pi(\bz)d\bz \\
& \leq ||\partial_{z_1}w||\, ||\partial_{z_1}\alpha_{I}\, w_{\eta\eta\eta}||    + 2 ||\partial_{z_1}\tilde U w||\, ||\partial_{z_1}w|| , 
\end{align*}
where integration by parts and the Cauchy-Schwarz inequality are used. 
Thus
\begin{equation}\label{wz_1}\frac{d}{dt} ||\partial_{z_1}w|| \leq \frac{1}{2}||\partial_{z_1}\alpha_{I}\, w_{\eta\eta\eta}|| + ||\partial_{z_1}\tilde U w|| . 
 \end{equation}

Clearly, to prove the boundedness of $||\partial_{z_1}w||$, it suffices to show the boundedness of the right hand side of \eqref{wz_1}. 
Fortunately, the right hand side of \eqref{wz_1} does not involve the $\mathbf z$ derivative of $w$. 
The estimates of the $\eta$ derivatives of $w$ are standard, which can be carried out in the following deductive way.

We calculate that 
\begin{align}
\label{W_eta}
\begin{split}
 \frac{d}{d t} ||w_{\eta}||^2 &= \int \bigl((w_t)_{\eta} \bar w_{\eta}+ w_{\eta} (\bar w_t)_{\eta} \bigr) d\eta\pi(\bz)d\bz \\ 
                &= \int \bigl (-i (\tilde U w)_{\eta} \bar w_{\eta}+ i w_{\eta} (\tilde U \bar w)_{\eta} \bigr)  d\eta\pi(\bz)d\bz \\
                & \le \left(\int (\tilde U w)_{\eta}^2 \, \bar w_{\eta}^2 \, d\eta\pi(\bz)d\bz\right)^{1/2}
                \left((\tilde U \bar w)_{\eta}^2 \, w_{\eta}^2\, d\eta\pi(\bz)d\bz\right)^{1/2} \\
                & = 2 ||w_{\eta}||\, ||(\tilde U w)_{\eta}||. 
 \end{split}
\end{align}
Since $\partial_{\eta}\tilde U= 4 \alpha_{I}\eta + \partial_{\eta}U_r / \e =O(1)$, and by the chain rule, 
$$|| (\tilde U w)_{\eta}|| = || \tilde U_{\eta} w + \tilde U w_{\eta}|| 
\leq C_1 ||w|| + C_2 ||w_{\eta}||, $$
where $C_1, C_2>0$ are constants. Thus 
\begin{equation}\label{w_eta1}\frac{d}{d t}||w_{\eta}|| \leq C_1 ||w|| + C_2 ||w_{\eta}||. 
\end{equation}
For $t \in [0,T]$, the boundedness of $||w_{\eta}||$ follows from the Gr\"onwall's inequality and assumption \eqref{assump2}.
Similarly, we get
\begin{align*}
\frac{d}{d t} ||w_{\eta\eta}||^2 & = \int \bigl ((w_t)_{\eta\eta} \bar w_{\eta\eta}+ w_{\eta\eta} (\bar w_t)_{\eta\eta}\bigr) d\eta\pi(\bz)d\bz \\
&= \int \bigl (-i(\tilde U w)_{\eta\eta} \bar w_{\eta\eta}+ i w_{\eta\eta} (\tilde U \bar w)_{\eta\eta}\bigr) d\eta\pi(\bz)d\bz \\
& \le 2 ||w_{\eta\eta}||\, ||(\tilde U w)_{\eta\eta}||, 
\end{align*}
where the Cauchy-Schwarz inequality is used again just as in (\ref{W_eta}). 
By the chain rule, $\exists\, C_1,C_2, C_3 \in \R$ such that
\[
 ||(\tilde U w)_{\eta\eta}|| \le C_1 ||w|| + C_2 ||w_{\eta}|| + C_3 ||w_{\eta\eta}||, 
\]
thus
\begin{equation} \label{w_eta2}
\frac{d}{d t} ||w_{\eta\eta}|| \le C_1 ||w|| + C_2 ||w_{\eta}|| + C_3 ||w_{\eta\eta}||. 
\end{equation}
The boundedness of $||w_{\eta\eta}||$ follows from Gr\"onwall's inequality and assumption \eqref{assump2}, and similarly
for $||w_{\eta\eta\eta}||$. $||\alpha_{I}w_{\eta\eta}||$, $||\tilde U w_{\eta}||$ are also bounded. 
By Lemma \ref{Newlemma:Ur}, %$||\partial_{z_1} \tilde U w|| \sim C \e^{\frac 3 2}||w||$. 
{\color{black}$||\partial_{z_1} \tilde U w|| \sim C \e^{\frac 1 2}||w||$.}
Using Gr\"onwall's inequality on (\ref{wz_1}), one gets $$||\partial_{z_1} w||\leq C_{T}, $$ 
where $C_{T}$ is a $O(1)$ constant, independent of $\e$. 
By induction, we obtain $$||\partial_{\bz}^{\bk}w||\leq C_{T}, $$
where $\partial_{\bz}^{\bk}=\partial_{z_1}^{k_1}\cdots \partial_{z_n}^{k_n}$. Therefore, we have shown Theorem \ref{w_z}. 
\end{proof}

%---------------------------------------------------------------
\section{Quantum and classical uncertainty}
\label{sec:3.5}
In this section we briefly discuss about quantum and classical uncertainty, and about the comparison between quantum and classical systems, for small values of the rescaled Planck's constant. For simplicity, we first consider the case with one degree of freedom, $x\in \mathbb{R}$, and scalar random variable $z$. 

\subsection{Moments and expectations}

A quantum system is completely determined by its wave function $\psi$. For each
realisation of the random variable $z$, the quantum system is described by $\psi(x,t,z)$. 

The primary physical quantities of interest include the position density, 
$$ \rho(t,x,z) = |\psi(t,x,z)|^2, $$
and the current density, 
$$ j(t,x,z) = \varepsilon\, \text{Im}\left(\overline\psi(t,x,z)\nabla\psi(t,x,z)\right). $$

Some quantities of interest to look at are the mean and standard deviation in $z$. In this way we can define the means:
\begin{align}
   \mathbb{E}[\rho](x,t) & = \int\rho(x,t,z) \pi(z)\, dz, \quad \mathbb{E}[j](x,t) = \int j(x,t,z) \pi(z)\, dz
   \label{eq:E1}
\end{align}
and variance: 
\begin{align}
   \mathbb{Var}[\rho](x,t) & = \mathbb{E}[\rho^2]-\mathbb{E}[\rho]^2, \quad
   \mathbb{Var}[j](x,t) = \mathbb{E}[j^2]-\mathbb{E}[j]^2.
   \label{eq:V1}
\end{align}
The standard deviation will be computed as the square root of the variance: 
\[
	\mathbb{SD}[\rho] = \sqrt{\mathbb{Var}[\rho]}, \quad \mathbb{SD}[j] = \sqrt{\mathbb{Var}[j]}.
\]

For quantum systems we denote by $<h> = <\psi|h|\psi>$ the expectation value of observable $h$.
Such a quantity will in general be a function of time and $z$. 
For example 
\begin{align*}
	<q> & = <q>(t,z) = <\psi|\hat{q}|\psi> = \int \bar{\psi}(x,t,z)x\psi(x,t,z)\,dx,\\
		<p> & = <p>(t,z) = <\psi|\hat{p}|\psi> = - i\varepsilon \int \bar{\psi}(x,t,z)\psi_x(x,t,z)\,dx,
\end{align*}
where $\hat{q} = x\cdot$ and $\hat{p} = -i\varepsilon\frac{\partial}{\partial x}$ denote, respectively, the position and momentum operators when the wave function $\psi$ is in the space representation.
 
Because of the uncertainty in the parameter $z$, such quantities are random variables. It is possible to compute mean and variance of them as a function of time: 
\begin{align*}
	\mathbb{E}[<h>] &= \int <h>(t,z)\pi(z)\,dz,\\
	\mathbb{Var}[<h>] & = \mathbb{E}((<h>-\mathbb{E}(<h>))^2) = \int (<h>-\mathbb{E}[<h>])^2\pi(z)\,dz, 
\end{align*}
where $h$ denotes, for example, $q$ or $p$.

Notice that the average density $\mathbb{E}[\rho]$ and the average current $\mathbb{E}[j]$ can be used to compute an (ensemble) average particle position and momentum, since the two integration processes commute. 
%i.e., for example, 
%$\mathbb{E}[<x>] = <\mathbb{E}[x]>$. 
However, the same is not true for the variance. As we shall see, it is possible to consider the classical limit of 
$\mathbb{Var}[<x>]$, while is it hard to define such a limit for $\mathbb{Var}[\rho]$.

\subsection{Classical limit}
\label{CL}
In classical mechanics, position and momentum of the particle follow Hamilton's equation
\[
	\dot{q} = \pad{H}{p}, \quad \dot{p} = - \pad{H}{q},
\]
subject to some initial condition $q(0,z) = q_0(z),\>p(0,z) = p_0(z)$. As in the quantum case, the uncertainty can be introduced at the level of the initial condition or at in the potential that defines the Hamiltonian:
\[
   H = \frac{p^2}{2m} + V(q,z),
\]
where the random parameter $z$ is distributed with a given density $\pi(z)$.

Position and momentum at a given time are therefore function of $z$ as well: $q = q(t,z), \> p = p(t,z)$.
If such a density is known, then the probability distribution function (pdf) of the coordinate $q$ can be found by classical techniques to find pdf of a function of a random variable. One which is commonly adopted in the physics community is given by
\begin{equation}
	\pdf_q(x,t) = \int \delta(x-q(t,z))\pi(z)\,dz,
	\label{classic_density}
\end{equation}
where $\delta$ denotes Dirac's delta, and the integral has to be interpreted in the usual distributional sense.
This representation can be interpreted as follows: for each realization or the random variable, 
a classical particle can be seen as a singular particle density
\[
	\rho_c(x,t,z) = \delta(x-q(t,z)).
\]
The probability distribution is then computed by weighting each value of the parameter with its probability density function, thus obtaining expression (\ref{classic_density}).

Assuming the function $q(t,z)$ is monotone in $z$, the integral can be easily computed by substitution, using the inverse function $z = z(t,q)$, yielding 
\[
	\pdf_q(x) = \pi(z(t,x)) |\partial z/\partial x|.
\]
A suitable generalisation is possible in the case $q(t,z)$ is not monotone:
\begin{equation}
	\label{pdf_q}
	\pdf_q(x) = \sum_{z:q(t,z) = x}\frac{\pi(z)}{|\partial q(t,z)/\partial z|}.
\end{equation}
Likewise, the mean current density distribution can be computed by smoothing the singular current corresponding to a single realization of the parameter $z$
\[
	 j_c(x,t,z) = \rho_c(x,t,z)p(t,z)
\]
by the pdf $\pi(z)$, obtaining
\[
	j_c(x,t) = \int \delta(x-q(t,z))p(t,z)\pi(z)\, dz.
\]
Using the same argument, such current distribution can be computed as
\[
	j_c(x,t) = \sum_{z:q(t,z) = x}\frac{\pi(z)}{|\partial q(t,z)/\partial z|}p(t,z).
\]

The situation with several degrees of freedom or with multivariate distribution is slightly different. 
Let us denote by $d$ the number of degrees of freedom, and by $m$ the number of random parameters. 
If $d=m$ then Equation (\ref{pdf_q}) is still valid by interpreting $|\partial q(t,z)/\partial z|$ as the Jacobian of the transformation between $z$ and $q$.
If $d>m$, then in general the pdf will be proportional to a Dirac mass on a manifold of dimension $d-m$. 
If $d<m$ then in general the pdf will still be a function, which can be computed by integration on a manifold of dimension $m-d$. As an example we mention here the case $d=1$, $m=2$. 
\begin{equation}
	\label{pdf_q2}
	\pdf_q(x) = \sum_{\Gamma:q(t,z) = x}\int_\Gamma\frac{\pi(z)}{|\nabla_zq(t,z)|}\, d\Gamma,
\end{equation}
where the sum is performed on all lines $\Gamma$ such that $q(t,z)=x$.

Sometimes one is not interested in the computation of the space distribution of the particle density or the current density, but just in some moments, such as the mean and the variance. They can be computed as 
\begin{align*}
	\mathbb{E}[q] & = \int q(t,z)\pi(z)\,dz,\quad
	 \mathbb{E}[p]  = \int p(t,z)\pi(z)\,dz,\\
	\mathbb{Var}[q] & = \mathbb{E}[q^2]-\mathbb{E}[q]^2, \quad 
	 \mathbb{Var}[p]  = \mathbb{E}[p^2]-\mathbb{E}[p]^2.
\end{align*}

%---------------------------------------------------------------------------------------
\section{Numerical Simulations}
\label{sec:4}

\subsection{The stochastic collocation method}
\label{sec:4a}

We now briefly review the gPC method \cite{XK-02}.
In its stochastic Galerkin formulation, the gPC-SG approximation has been successfully applied to
many stochastic physical and engineering problems, see for instance an overview \cite{Xiu, LK}. 

On the other hand, the stochastic collocation (SC) method \cite{XH, GWZ-14} is known as a popular choice for complex systems with 
uncertainties when reliable, well-established deterministic solvers exist. 
It is non-intrusive, so it preserves all features of the deterministic scheme, and easy to parallelize \cite{Xiu, NTW-08}. 
The basic idea is as follows. Let $\{\bz_k\}_{k=1}^{N_z}\subset I_{\bz}$ be the set of collocation nodes, $N_z$ the number of samples.
For each fixed individual sample $\bz_k$, $k=1,\ldots,N_z$, 
one applies the deterministic solver to the deterministic equations as in \cite{GWPT}, 
obtains the solution ensemble for a general function $f(t,x,\bz)$, $f_k(t,x)=f(t,x,\bz_k)$, 
then adopts an interpolation approach to construct a gPC approximation
$$f(t,x,\bz)=\sum_{k=1}^{N_z}f_k(t,x)l_k(\bz),$$ where $l_k(\bz)$ depends on the construction method. 
The Lagrange interpolation is used here by choosing $l_k(\bz_i)=\delta_{ik}$. With samples $\{\bz_k\}$ and 
corresponding weights $\{\nu_k\}$ chosen from the quadrature rule, the integrals in $\bz$ are approximated by 
\begin{equation}\label{Int-z} \int_{I_{\bz}}f(t,x,\bz)\pi(\bz)d\bz \approx \sum_{k=1}^{N_z}f_k(t,x)\nu_k, \end{equation}
where $$\nu_k = \int_{I_{\bz}}l_k(\bz)\pi(\bz) d\bz. $$
{\color{black}
Note that in practice Lagrange interpolation is used here only in order to construct the weights of a quadrature formula, once the nodes are assigned. }

Considering the structure complexity of the deterministic solver developed in \cite{GWPT,GWPT2}, 
we choose the SC rather than gPC-SG in this project due to its 
simplicity and efficiency in implementation. In numerical simulation, we need to sample the Gaussian wave packet parameters and the $w$ function, while in solution construction, we are interested in the wave equation as well as physical observables which are nonlinear transforms of the wave equation. In this sense, the stochastic collocation method offers great flexibility in computing averages of various quantities in the random space.
We will introduce below how the SC and the GWPT method are combined in our numerical implementation. 

%-----------------------------------------------------
\subsection{Numerical implementation}

We first briefly review the meshing strategy for the GWPT based method in $x$ and $t$ when the random variables are not present. Recall that the GWPT maps the semi-classical Schr\"odinger  equation to the ODE system for the wave packet parameters and the $w$ equation. In numerical simulation, we denote the time step for the ODE system by $\Delta t_1$, the time step for the $w$ equation by $\Delta t_2$ and the spatial grid size for the $w$ equation by $\Delta \eta$.
We introduce a spatial mesh in $\bx$ with grid size $\Delta x$, in the final reconstruction step for $\psi$. 

Since the $w$ equation does not produce $O(\e)$ scaled oscillations, $\Delta t_2$ and $\Delta\eta$  can be chosen independently of $\e$. 
The phase term in the GWPT is computed by solving the ODE system of the Gaussian wave packet parameters, where the ODE system does not contain stiff terms due to $\e$. However, the numerical error in solving the phase term is magnified by a factor of $\e^{-1}$ when constructing the wave function, thus we often need to take $\Delta t_1$ to be $O(\e^{1/k})$, where $k$ denotes the accuracy order of the ODE solver. Finally, $\Delta x$ used in the reconstruction step also needs to be $O(\e)$ in order to resolve small oscillations in the wave function. 
The interested readers may refer to \cite{GWPT,GWPT2} for a more detailed discussion. 

We now discuss the sampling strategy for each step of the numerical implementation in the random space. 
Three sets of collocation points in $\bz$ are used in our numerical tests:  
\begin{itemize}
\item number of points to solve the ODEs for the wave packet parameters given by the following system (\ref{psys_z}): $N_{z,1}$, \\
At each collocation point ${\bz}_k$ ($k=1, \cdots, N_{z,1}$), we have
\begin{equation} 
\label{psys_z}
\left\{ \begin{array}{l}
\partial_t q(t,{\bz}_k) = p(t,{\bz}_k), 
\\[2pt]\partial_t p(t,{\bz}_k) =  -\nabla V(q,{\bz}_k), 
\\[2pt] \partial_t\alpha(t,{\bz}_k) = -2 (\alpha(t,{\bz}_k))^2 -\frac{1}{2}\nabla\nabla V(q,{\bz}_k), 
\\[2pt] \partial_t\gamma(t,{\bz}_k) =\frac{1}{2}(p(t,{\bz}_k))^2 -V(q,{\bz}_k)+i\e\, {\rm Tr}(\alpha_{R}(t,{\bz}_k)). 
\end{array}\right.
\end{equation}
\item number of points to solve the $w$ equation (\ref{w_t}): $N_{z,2}$,  
\item number of points to reconstruct $\psi$: $N_{z,3}$.  
\end{itemize}
The sets of mesh points $N_{z,1}$, $N_{z,2}$ and $N_{z.3}$ used above are denoted by $M_{1}$, $M_{2}$ and $M_{3}$ respectively. 
The cardinality of the set $M_j$ is $N_{z,j}$ ($j=1,2,3$). {\color{black} How we choose these collocation points depends on the distribution of $\bz$. The correspondence between the type of polynomial chaos and their underlying random variables can be found in \cite{Xiu}. 
In our numerical tests, if $\bz$ is uniformly distributed on $[-1,1]$, the Legendre-Gauss quadrature nodes and weights are used; 
if $\bz$ follows the Gaussian distribution, the Gauss-Hermite quadrature rule is applied. }

{\color{black} We solve the ODE system (\ref{psys_z}) by using the fourth-order Runge-Kutta method.}
After computing the wave packet parameters in (\ref{psys_z}) at the mesh points $M_{1}$ in the $z$ direction, 
cubic spline interpolation is used to get the values of these parameters at the mesh points $M_{2}$, 
which prepares us to update $w$ by solving (\ref{w_t}) in time at the same mesh points of $z$, i.e., $M_{2}$. 
In the reconstruction step, for each ${\bz}_k$, 
\begin{equation} \label{eq:reconstruct}
\psi(x, t, {\bz}_k)=\widetilde w(\xi, t, {\bz}_k)\exp\left(i\, (\xi^{T} \alpha_R\, \xi +p^{T}\xi+\gamma)/\e \right),
\end{equation}
with $\xi=x-q$ and ${\bz}_k\in M_3$. Cubic spline interpolation is used to obtain values of wave packet parameters from $M_{1}$ to $M_{3}$, and 
values of $w$ from $M_{2}$ to $M_{3}$, which is the reconstruction mesh points. 
Now we have the values of $\psi=\psi(x, t, {\bz}_k)$ at mesh points $\{{\bz}_k\}$, for $k=1, \cdots, N_{z, 3}$. 
Finally, in order to plot the solution $\psi$ and its physical quantities of interest at a set of fixed physical location, denoted by $X_0$, 
cubic spline interpolation is used again. 
Note that in practice one does not need to reconstruct $\psi$: the observables, such as
the expectation values for position and momentum, can be computed directly from $w$. See Section 2.4 in \cite{GWPT}. 

We now discuss the sampling strategy in $\bz$ for the GWPT method. Although the parameter system 
\eqref{psys_z} does not have $\e$ oscillations in $\bz$, the numerical error in the parameters are magnified by $\e^{-1}$ when reconstructing the wave function as in \eqref{eq:reconstruct}. Therefore, one expects that $N_{z,1}$ should depend on $\e$
in order to obtain accurate approximation of the wave function, and the dependence is related to the ODE solver used for the parameter system. 
Due to the regularity property of the $w$ equation in $\bz$, we expect that we can take $\e$ independent numbers of collocation points, and thus it suffices to take $N_{z,2}=O(1)$. Clearly, we need to take $N_{z,3}=O(\e^{-1})$ to resolve the oscillation in $\bz$ in the wave function reconstruction. 

We remark that, although in the sampling strategy in $z$ we require $N_{z,1}$ and $N_{z,3}$ to be $\e$ dependent, it does not cause much computational burdens, because the $N_{z,1}$ collocation points are only used to the ODE system and the $N_{z,3}$ collocation points are used in the final step. On the other hand, {\it in the most expensive part, solving the $w$ equation in time, we only use $\e$ independent numbers of collocation points. Hence, such sampling strategy is desired for the sake of computational efficiency.}

Finally, considering  that certain physical observables, such as position density and flux density can be obtained directly from the $w$ function, then $N_{z,1}$ is expected to be independent of $\e$ if one is only interested to capture the correct physical observables. 
This argument will be verified numerically in the tests of Section \ref{sec:numtest}. 

%-------------------------------------------------------------
\subsubsection{Numerical observables}

With the obtained $\psi(t,X,{\bz}_k)$ and the corresponding weights $\{\nu_k\}$ for $k=1, \cdots, N_{z, 3}$, 
chosen from the quadrature rule, one can approximate the integral given by (\ref{energy}), 
\begin{equation}\label{E_approx} ||\psi||_{\Gamma}^2 =\int_{I_{\bz}}\int \left|\psi(t, \mathbf x, \bz)\right|^2 d{\mathbf x}\pi(\bz)d\bz 
\approx \sum_{k=1}^{N_{{z}, 3}}\sum_{i=1}^{N_x} |\psi(x_i, {\bz}_k)|^2\Delta x\, \nu_k. 
\end{equation}

To be consistent with the $\Gamma$-norm we defined in (\ref{energy}), 
we first denote the $L^1(x)$ norm of $j(t,x,\bz_k)$ by $\widetilde j (t,\bz_k)$ for $k=1, \cdots N_{z,3}$, 
\[
   \widetilde j (t,\bz_k) = \e \int \text{Im} \left( \overline\psi(t, x, \bz_k)\nabla\psi(t, x, \bz_k)\right) dx \approx
\sum_{i=1}^{N_x} \e \left( \overline\psi(t, x_i, \bz_k)\nabla\psi(t, x_i, \bz_k)\right) \Delta x,
\]
then get $\mathbb {E}(\widetilde j)$, $\mathbb {Var}(\widetilde j)$ and $\mathbb{SD}(\widetilde j)$: 
\begin{equation}
	\label{CD} \mathbb E(\widetilde j)\approx \sum_{k=1}^{N_{{z}, 3}} \widetilde j(t,\bz_k) \nu_k, \qquad
	\mathbb {Var}(\widetilde j) = \mathbb E(\widetilde j^2) - (\mathbb E(\widetilde j))^2, \qquad
	\mathbb{SD}(\widetilde j) = \sqrt{\mathbb{Var}(\widetilde j)}. 
\end{equation}

\subsubsection{Definition of errors}
The error in $\psi$ is computed by comparing $\psi$ with the reference solutions obtained from the second-order direct splitting method, 
where a sufficiently large number of collocation points $N_{z, 4}$ is used. Denote the set of mesh points by $M_{4}$. 
The error is measured under the averaged norm (\ref{energy}),
with the discretized form of the approximation shown in (\ref{E_approx}). More precisely, we use the relative error defined by 
\begin{equation}\label{Er_psi} \text{Er}[\psi] = \frac{||\psi_{G} - \psi_{D} ||_{\Gamma}}{||\psi_{D} ||_{\Gamma}}, 
\end{equation}
where $\psi_{G}$ represents the solution obtained by the GWPT method, and $\psi_{D}$ represents the one obtained from the direct splitting method. 
To compute $||\psi_{G} - \psi_{D} ||_{\Gamma}$ in (\ref{Er_psi}), 
one needs values of both $\psi_{G}(t, X_0, z_j)$ and $\psi_{D}(t, X_0, z_j)$ at the same set of mesh points of $z$, denoted by 
$M_{5}$ and the number of collocation points $N_{z, 5}$. Thus 
$$ ||\psi_{G} - \psi_{D}||_{\Gamma}^2 \approx \sum_{j=1}^{N_{{z}, 5}}\sum_{i=1}^{N_x} |\psi_{G}(x_i, {z}_j) - \psi_{D}(x_i, {z}_j)|^2 \Delta x\, \nu_j. $$
Here one needs to find the interpolated values of $\Psi_{G}$ and $\Psi_{D}$ corresponding to the mesh points 
$M_{5}$, from mesh points $M_{3}$ and $M_{4}$ respectively. 
We let $M_5=M_3$ for simplicity. All the interpolation in $z$ space refers to the spline interpolation. 

To quantify the errors in mean and standard deviation of the current density, we use
\begin{equation}
\label{Er_j}
\text{Er}_1 [j]  = \left|\frac{{\mathbb E}(\widetilde j_{G} - \widetilde j_{D})}{{\mathbb E}(\widetilde j_{D})}\right|, \qquad
\text{Er}_2 [j]  = \left|\frac{\mathbb {SD}(\widetilde j_{G} - \widetilde j_{D})}{\mathbb{SD}(\widetilde j_{D})}\right|, 
\end{equation}
where $\mathbb{E}$ and  $\mathbb{SD}$ are calculated by using (\ref{CD}). 
Note that the mass at each collocation point $\bz_k$ namely 
$\int |\psi(t,x,\bz_k)|^2 dx$ is a conserved quantity with respect to time and a constant. 
Thus it is not interesting to compute the relative errors of $\rho$ in terms of the $\Gamma$-norm as that for $j$. 

%------------------------------------------------------------------------------

\subsection{Numerical Tests} 
\label{sec:numtest}

{\bf Part I: Relation between $N_{z,1}$, $N_{z,2}$, $N_{z,3}$, $N_{z,4}$ and $\e$ }

We know from the deterministic problem in \cite{GWPT} that all the stiffness in time and space of the original Schr\"odinger equation for $\psi$ associated with very small values of $\e$ is essentially been removed in the equation for $w(\eta,t)$ by the GWP transform. 
Since we proved in subsection \ref{Reg-W} that $w$ equation is not oscillatory in $z$, 
all the orders of $z$-derivatives of $w$ have a uniform upper bound, thus the number of collocation points used to solve $w$, namely $N_{z,2}$
is expected to be independent of $\varepsilon$. 

In Part I of the numerical tests, we will demonstrate for sufficiently large $N_{z,1}$ and $N_{z,3}$ that are proportional to 
$1/\e$ and $1/\sqrt{\e}$ respectively, one can choose $N_{z,2}$ {\it uniformly} with respect to $\e$. 
\\[2pt]

We put uncertainty in the potential function in Test (a1)--(a3).

{\bf Test (a1)} 

In this test, we assume $z$ a one-dimensional random variable that follows a uniform distribution on $[-1,1]$. Consider the spatial domain $x\in[-\pi, \pi]$
with periodic boundary conditions. 
The initial data of $\psi$ is
\begin{equation}\label{Psi-IC} \psi(x,0)=A \exp\bigl[(i/\e)\bigl(\alpha_0 (x-q_0)^2 + p_0 (x-q_0)\bigr)\bigr], \end{equation}
where $q_0=\pi/2, p_0=0, \alpha_0=i$. 
We name Test (a1-i) and Test (a1-ii) with different potentials: 
\begin{align*}
& \text {{\bf Test  (a1-i)}} \qquad   V(x,z)=(1+0.95z)x^2,  \\[4pt]
& \text {{\bf Test (a1-ii)}}  \qquad  V(x, z) = (1-\cos(x))(1+0.9z). 
\end{align*}
Let $\Delta t_1=0.01$, $\Delta\eta=0.3125$ in the $w$ equation, and $\Delta t_2=2.5\times 10^{-4}$ in the ODEs
for solving the parameters $p$, $q$, $\alpha$, $\gamma$.  $\Delta t=1/600$ and $\Delta x=2\pi/9600$ in the reference solutions. 
$T=1$ in Test (a1)--(a3). 

\begin{figure}[H]
\centering
\includegraphics[width=1.0\linewidth]{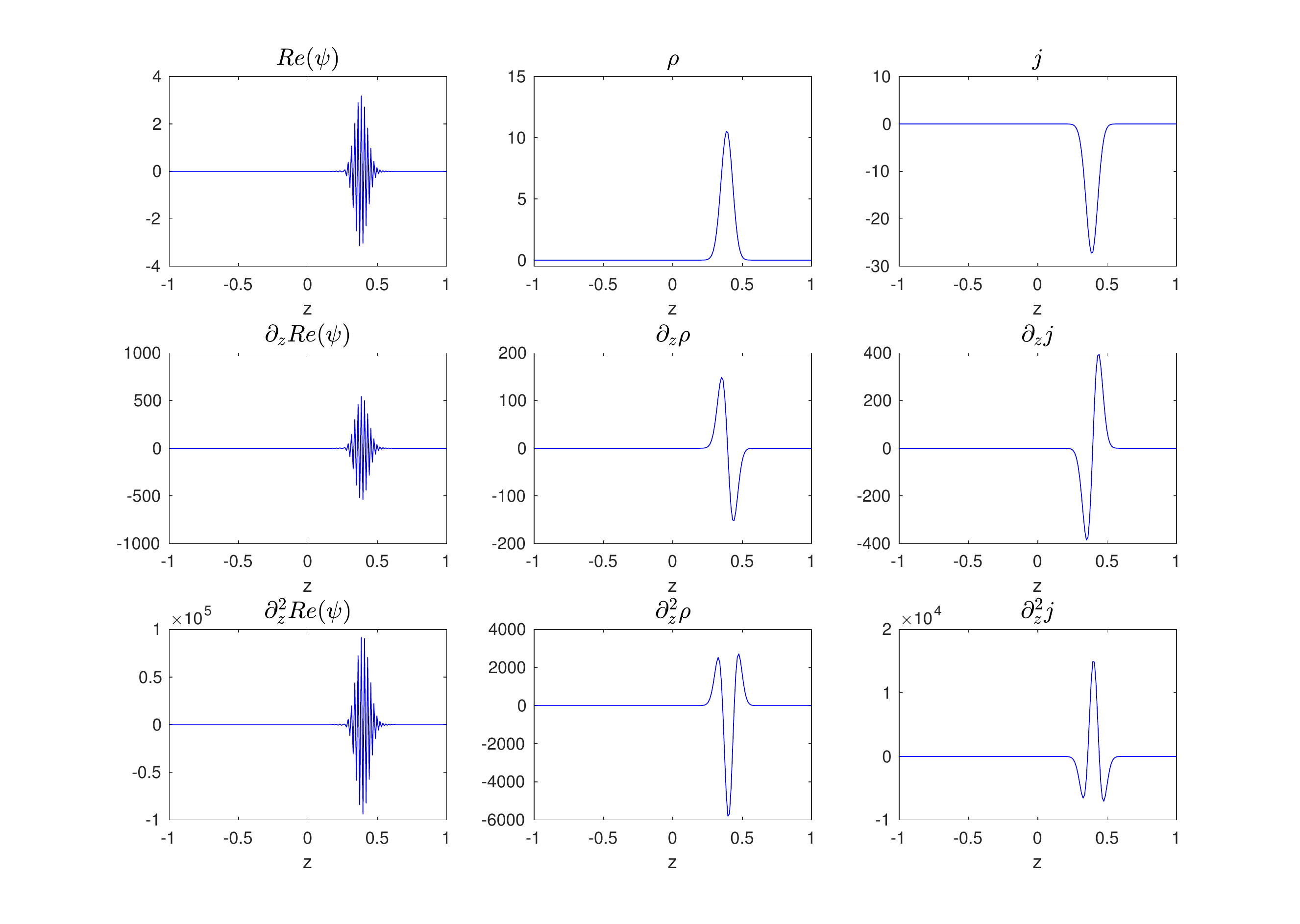}
\caption{Test (a1-i), $\varepsilon=1/256$. Plot at $x=-0.1316$, $T=1$. }
\label{TestI-a4}
\end{figure}

In Test (a1-i), we choose the random potential $V(x,z)=(1+0.95z)x^2$. 
The reason we choose potential function quadratic in $x$ is that the form of exact solution (\ref{Heller}) is known in this case. 
Solutions of $\text{Re}(\psi)$, $\rho$, $j$ as a function of $z$ at some physical location $x$ are plotted in Figure \ref{TestI-a4}. 
The plot indicates that 1) the behaviors of $\psi$ or $z$-derivatives of $\psi$ are much more oscillatory than that of $\rho$ and $j$; 
2) regarding the relations between the amplitude of $\text{Re}(\psi)$, $\rho$, $j$ and their first and second order 
partial $z$-derivatives, numerical results seem to suggest that $\text{Re}(\psi)$ increases by $O(1/\e)$ each time we differentiate it, 
and $\rho$, $j$ tend to increase by $O(1/\sqrt{\e})$ as the order of $z$-derivative increases. 
\\[2pt]

%------------------------------------------------------------------------------
We now compare the trend of maximum values of the $z$-derivatives of $\psi$ and $w$ with respect to different $\e$: 
\begin{table}[H]
\begin{center}
\begin{tabular}{ |p{2.5cm} | p{2cm} | p{2cm}| p{2cm}| p{2cm}| p{2cm}| p{2cm}| }
 \hline
   $\varepsilon$  &  $\frac{1}{32}$ & $\frac{1}{64}$ &  $\frac{1}{128}$ & $\frac{1}{256}$ &  $\frac{1}{512}$ \\
 \hline 
  $\max\partial_z Re(\psi)$ &  $40.5971$ & $96.1058$ &  $228.0143$ & $541.3633$ &  $1.2843e+03$ \\ 
   \hline 
    $\max\partial_z Re(w)$ & $0.0618$ &  $0.0521$ & $0.0439$ & $0.0370$ & $0.0313$  \\ 
     \hline 
     \end{tabular}
    \caption{Test (a1-ii). Comparison of $\max\partial_z Re(\psi)$ and $\max\partial_z Re(w)$. }
    \label{Comp}
   \end{center}
\end{table}
One can observe from Table \ref{Comp} that $\max\partial_z \text{Re}(\psi)$ increase much more rapidly than $\max\partial_z \text{Re}(w)$, 
the former doubles its values as $\e$ decreases to half smaller, the latter slightly changes its values. This demonstrates that while $\psi$ is highly oscillatory in the random space, $w$ is smooth in $z$, and it guides us to choose the following number of collocation points: 
take sufficiently large $N_{z, 1}=N_{z, 3}=N_{z,4}=500$ and 
$N_{z, 2}=32$ in Test (a1)--(a3). 

\begin{figure}[H]
\includegraphics[width=0.49\linewidth]{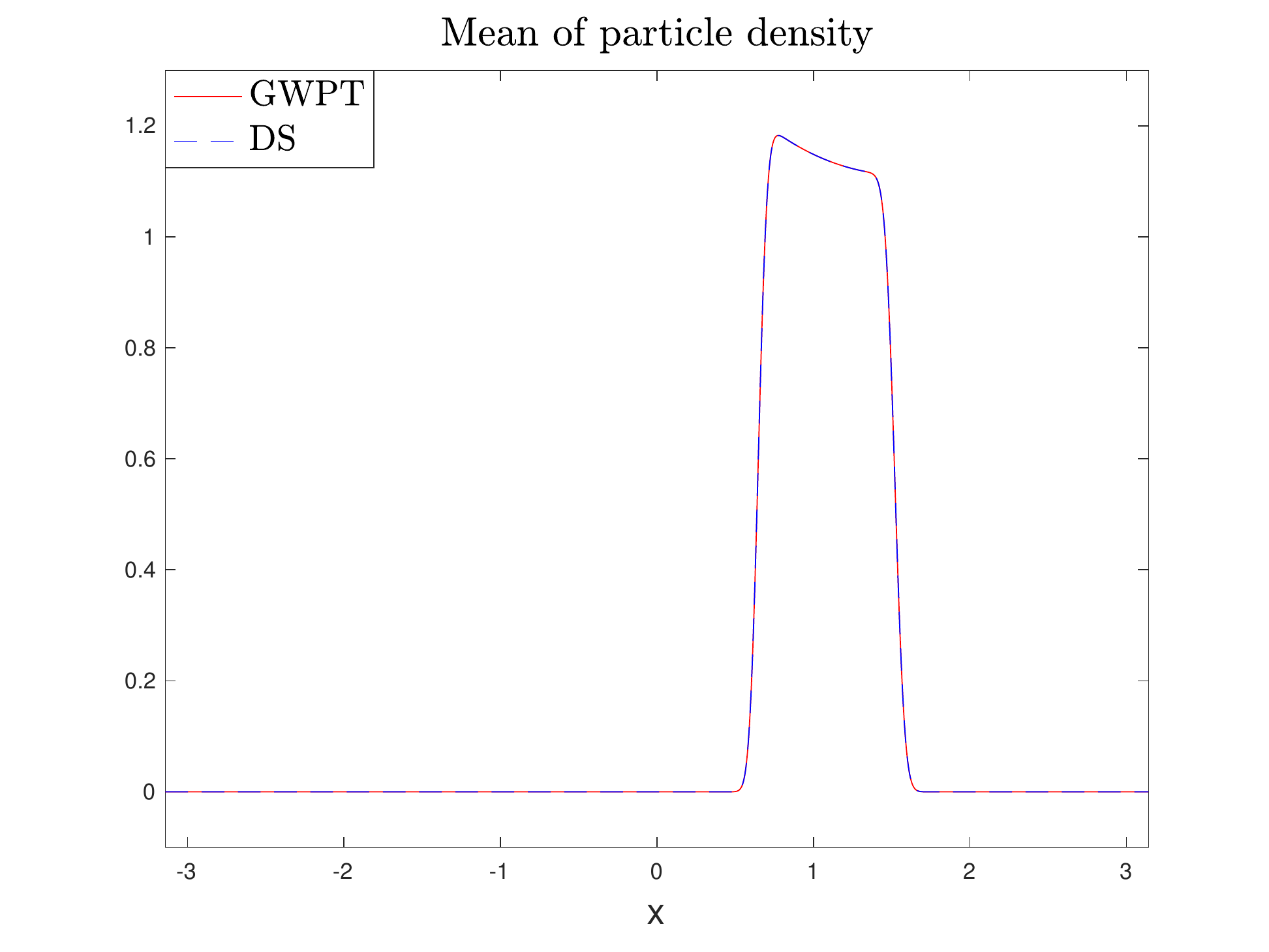}
\includegraphics[width=0.49\linewidth]{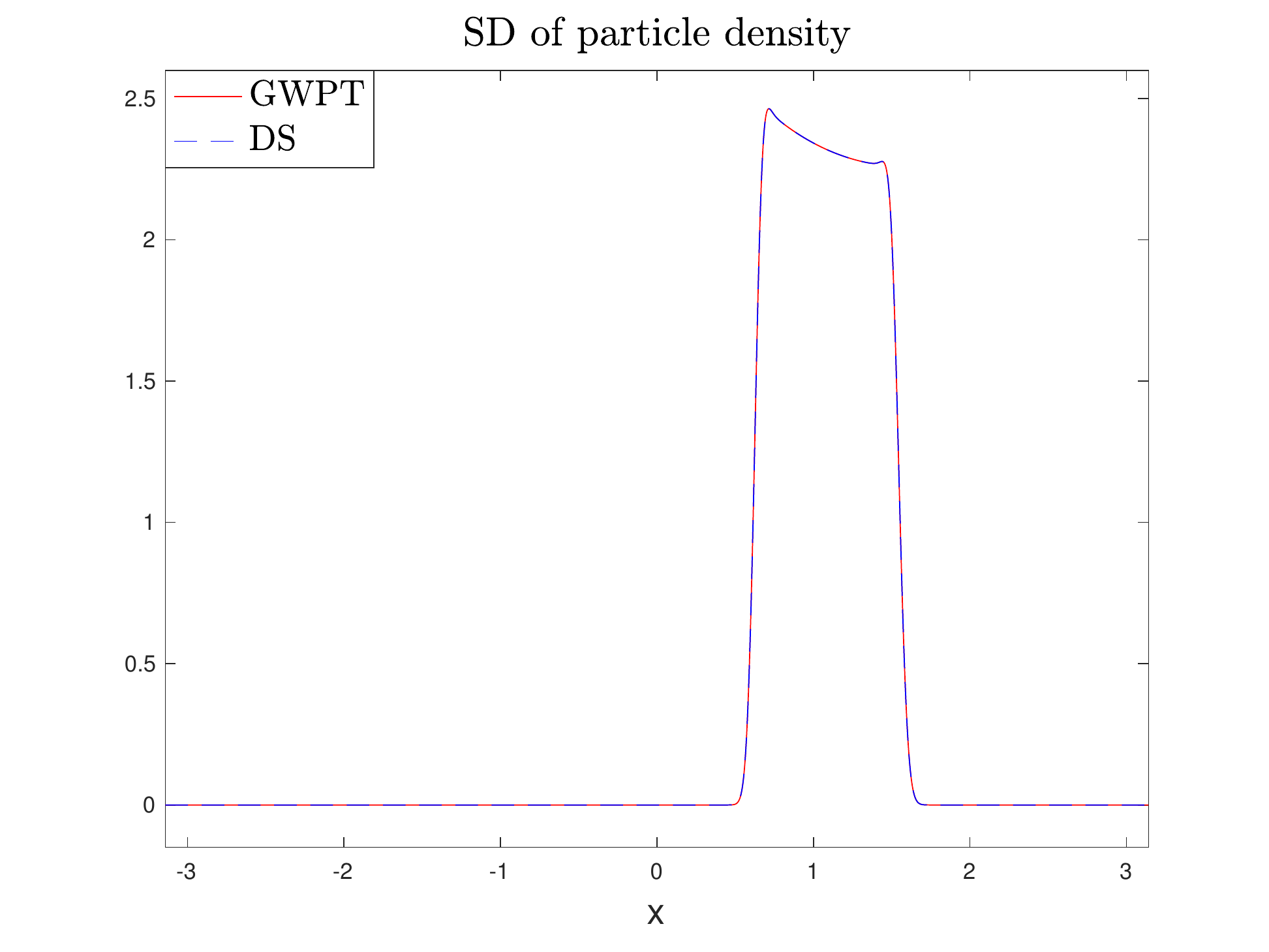}
\includegraphics[width=0.49\linewidth]{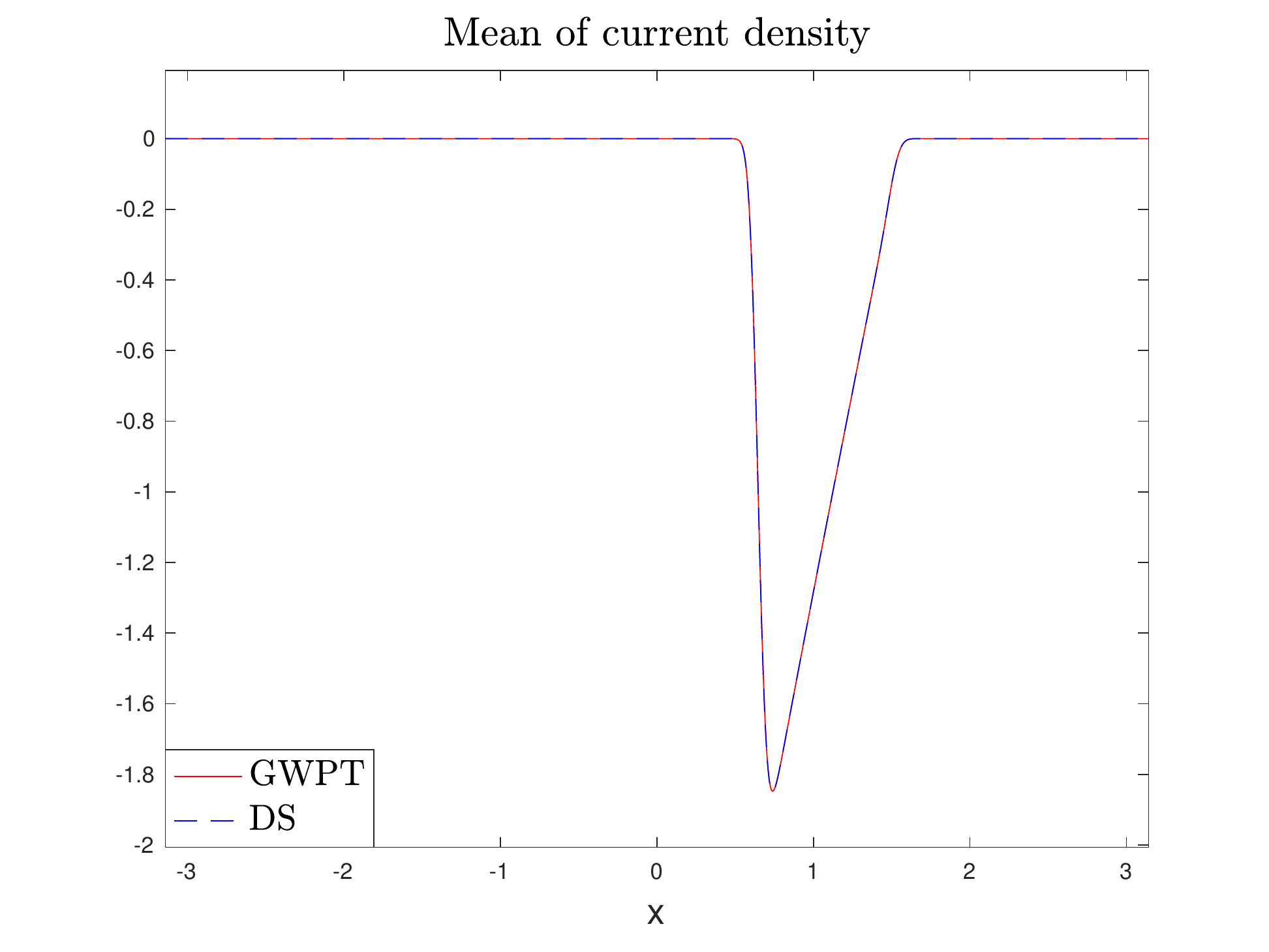}
\includegraphics[width=0.49\linewidth]{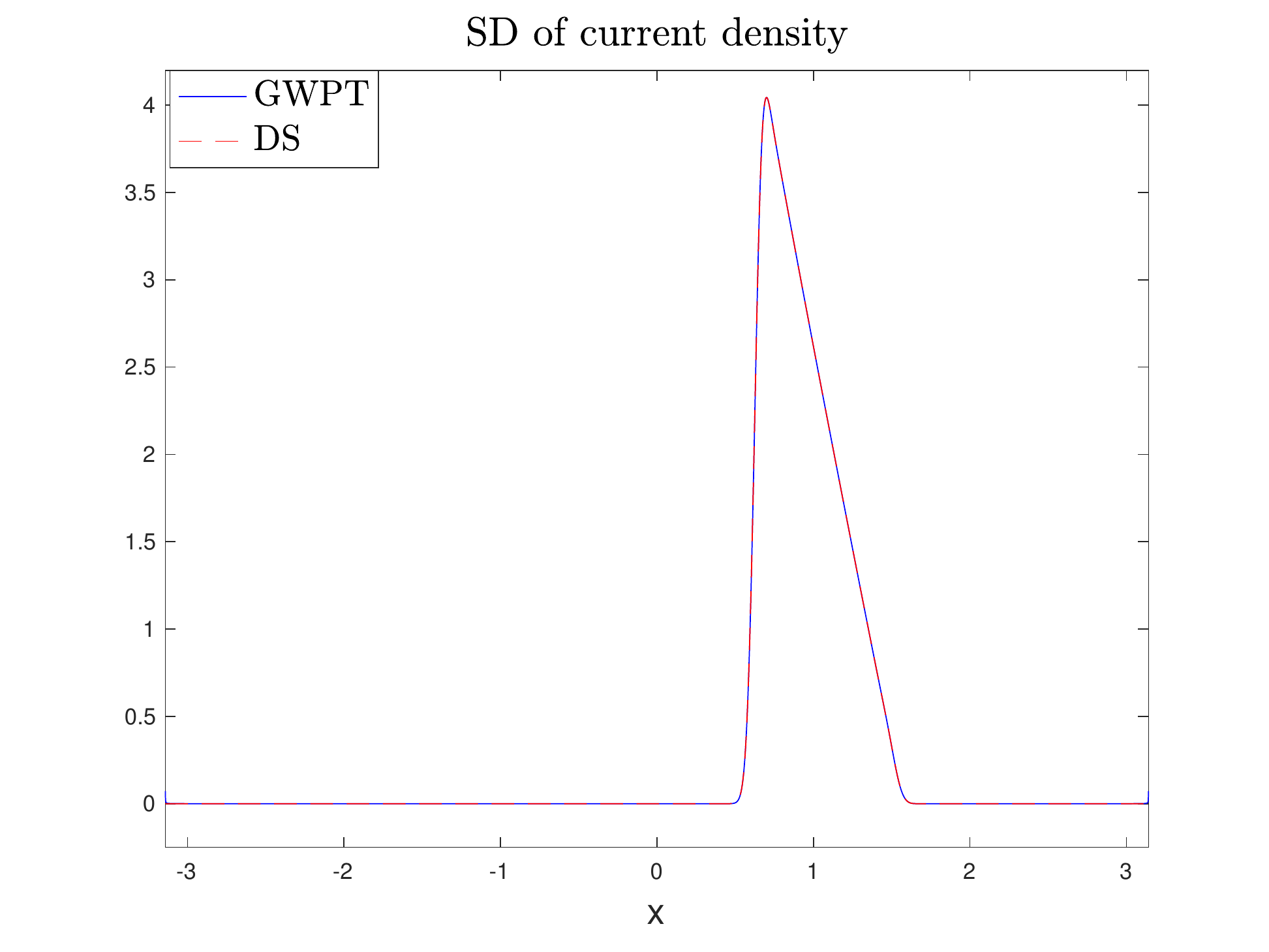}
\includegraphics[width=0.49\linewidth]{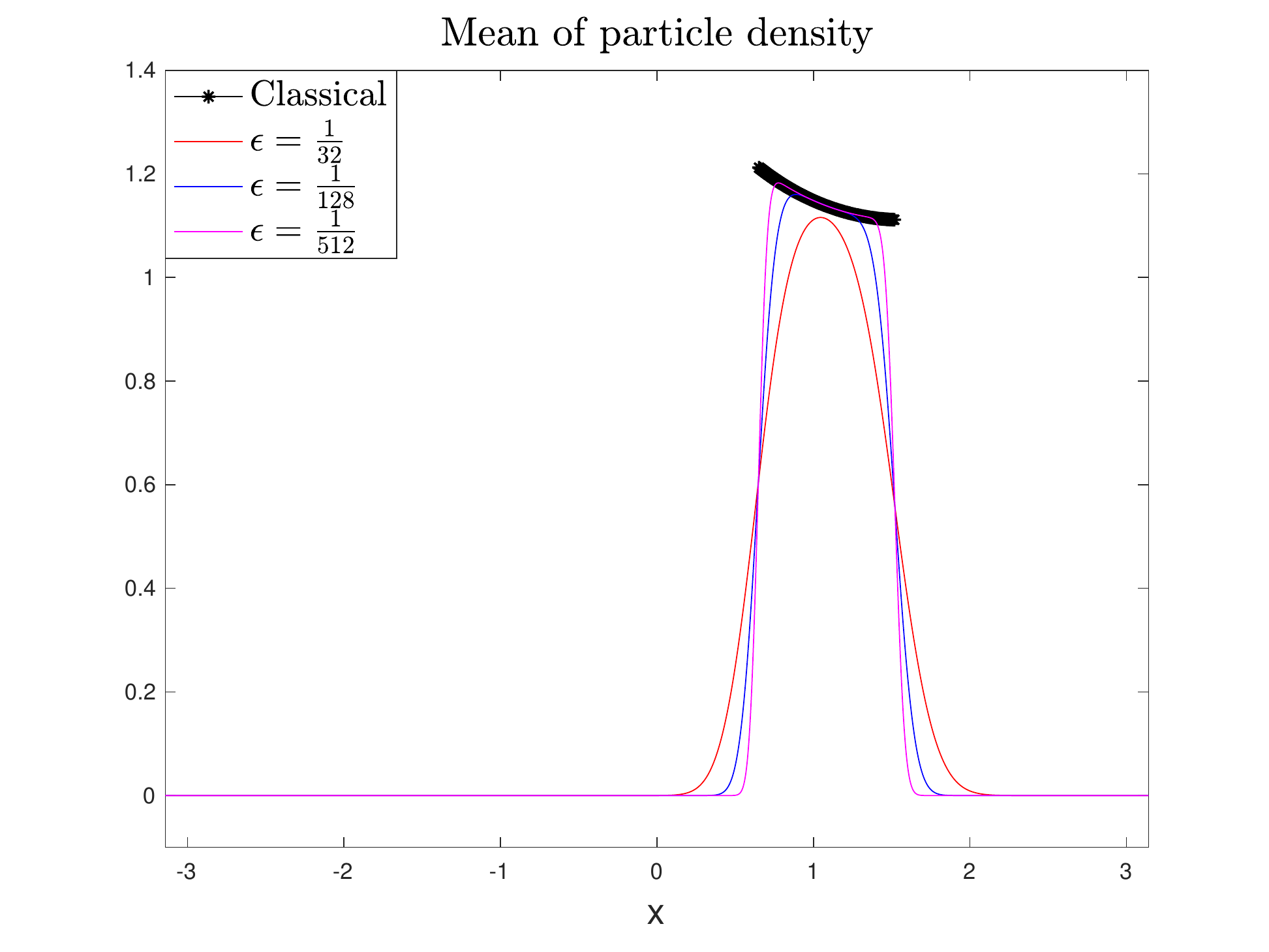}
\includegraphics[width=0.49\linewidth]{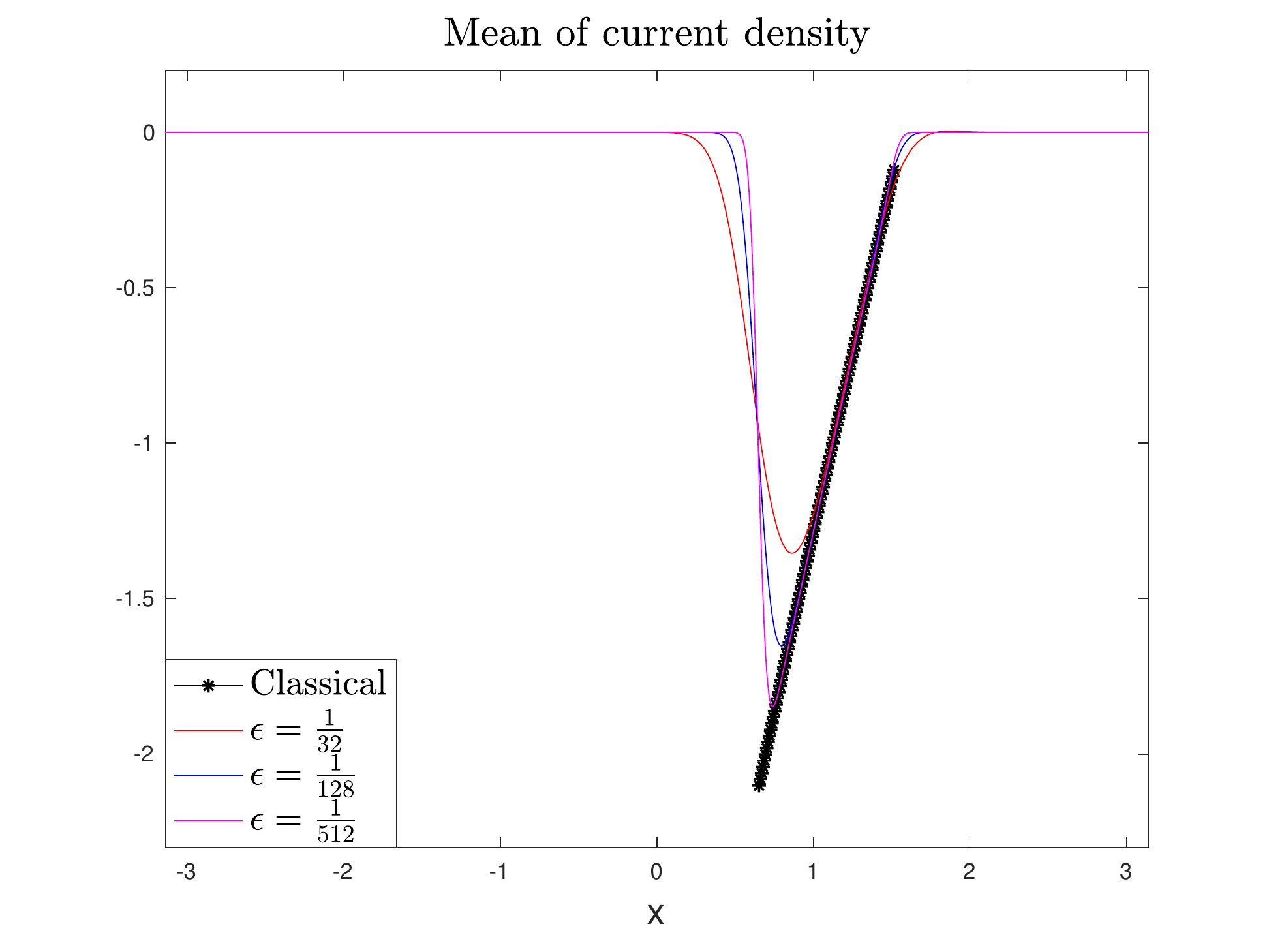}
\caption{Test (a1-ii). Comparison of the GWPT, DS and Classical method. $\e=\frac{1}{512}$ for the first two rows. }
\label{I-a1}
\end{figure}
%% save in 'Fig2_Data.mat' (eps=1/256), 'Fig2_Eps512.mat' (eps=1/512), 'Fig2_Eps32.mat' (eps=1/32)

\begin{table}[H]
\begin{center}
\begin{tabular}{ |p{1.2cm} | p{1.8cm} | p{1.8cm}| p{1.8cm} | p{1.8cm}|}
\hline
$\e$ & \text{GWPT} &\text{DS} & $N_{z,1}$ & $N$ \\
%\hline
%$\frac{1}{32}$ & $8.8$ & $1.2$ & $100$ & $100$  \\
%\hline
%$\frac{1}{128}$ & $16.3$ & $6.5$ & $200$ & $200$  \\
\hline
$\frac{1}{256}$ & $30.2$  & $40.6$  & $400$ & $400$    \\
\hline
$\frac{1}{512}$ & $44.4$ & $123.7$  & $600$ & $600$   \\
\hline
$\frac{1}{640}$ & $60.2$  & $286.3$ & $800$ & $800$ \\
\hline
$\frac{1}{768}$ & $70.6$ & $450.5$ & $900$  & $900$ \\
 \hline
 \end{tabular}
  \caption{
  Comparison of CPU time (in seconds) of using the GWPT and DS with different $\e$. $\Delta t_1$, $\Delta t_2$, $\Delta\eta$ in GWPT are the same as in Test (a1-i). $N_{z,2}=32$, 
  $N_{z,1}=N_{z,3}=N_{z,4}$ and $\Delta t=1/N$, $\Delta x = 2 \pi/6N$ in the DS method. }
  \label{T0}
     \end{center}
\end{table}

\begin{table}[H]
\begin{center}
\begin{tabular}{ |p{1cm} | p{2cm} | p{2cm}| p{2cm}| p{2cm}| p{2cm} | }
 \hline
  $\varepsilon$  & $\text{Er}[\psi]$   &  $\text{Er}_1[j]$  &  $\text{Er}_2[j]$      \\
 \hline
 $\frac{1}{32}$ & $2.9968e-05$  &  $3.1235e-07$ & $1.0328e-06$  \\
 \hline
     $\frac{1}{64}$ & $3.2442e-05$ & $3.6868e-07$ & $9.5517e-07$   \\
      \hline
      $ \frac{1}{128}$ & $3.1316e-05$ & $1.5865e-07$ & $4.2628e-07$     \\   
 \hline
         $\frac{1}{256}$ &  $3.7924e-05$ & $1.6370e-07$ & $3.8290e-07$   \\   
          \hline
         \end{tabular}
    \caption{Test (a1-ii). Error of $\psi$ and mean and standard deviation of $j$ with respect to different $\varepsilon$. }
      \label{T1}
   \end{center}
\end{table}

%--------------------------------------------------------------
\begin{table}[H]
\begin{center}
\begin{tabular}{| p{1cm} | p{2cm} | p{2cm} | p{2cm} | p{2cm} |p{2cm} |}
 \hline
  $N_{z,2}$  & $\text{Er}[\psi]$  &  $\text{Er}_1[j]$   &  $\text{Er}_2[j]$    \\
   \hline
   $2$  & $0.0013$ &   $3.8978e-07$ & $5.4256e-05$   \\
   \hline
   $4$  & $1.7219e-05$ &  $7.4463e-09$ & $3.0955e-06$ \\
    \hline
   $8$  & $1.3416e-07$ &  $6.4828e-09$ & $4.1892e-08$ \\
     \hline
     $16$  & $7.5591e-09$ & $4.5480e-10$ & $1.4003e-09$  \\
      \hline
     $32$  &  $4.9371e-10$  &  $2.9899e-11$ & $8.8615e-11$ \\
    \hline
    \end{tabular}
\caption{Test (a1-ii), $\e=1/256$. Relative errors for solutions computed by increasing $N_{z,2}$. }
  \label{T2}
 \end{center}
\end{table}

%--------------------------------------------------------------

Test (a2)--(a3) use the same initial data of $\psi$ as shown in (\ref{Psi-IC}). 
\\[2pt]

{\bf Test (a2)}  Let $z$ follow the normal Gaussian distribution, with the random potential 
$$ V(x, z) = (1-\cos(x))(1+0.9z). $$ 

\begin{table}[H]
\begin{center}
\begin{tabular}{ |p{1cm} | p{2.5cm} | p{2.5cm}| p{2.5cm}|}
 \hline
  $\varepsilon$  & $\text{Er}[\psi]$    &  $\text{Er}_1[j]$   &  $\text{Er}_2[j]$      \\
 \hline
   $\frac{1}{32}$ & $4.8231e-05$ &  $2.4561e-07$ & $1.8024e-05$    \\
   \hline
     $\frac{1}{64}$ & $3.0519e-05$ &  $2.3575e-07$ &  $8.7407e-06$  \\ 
    \hline
      $ \frac{1}{128}$  & $3.1520e-05$  &  $2.3147e-07$ & $6.1686e-06$  \\   
    \hline
         $\frac{1}{256}$  &  $3.7955e-05$ & $2.1877e-07$ & $5.3444e-06$   \\
          \hline
         \end{tabular}
    \caption{Test (a2). Error of $\psi$ and mean and standard deviation of $j$ with respect to different $\varepsilon$. }
     \label{T3}
   \end{center}
\end{table}

%------------------------------------------------------------------------------

{\bf Test (a3)}
We let \begin{equation}\label{V1} V(x, z) = (1-\cos(x))(1+ \varepsilon z), 
\end{equation}
with $z$ followed the normal Gaussian distribution. 

\begin{table}[H]
\begin{center}
\begin{tabular}{ |p{1cm} | p{2cm} | p{2cm}| p{2cm}| p{2cm}| p{2cm} | }
 \hline
  $\varepsilon$   & $\text{Er}[\psi]$   &  $\text{Er}_1[j]$   &  $\text{Er}_2[j]$      \\
 \hline
   $\frac{1}{32}$ & $2.9898e-05$ &   $4.9752e-08$ & $8.3984e-07$  \\
    \hline
     $\frac{1}{64}$ &  $3.2074e-05$ &  $1.4983e-07$ & $7.9849e-07$  \\ 
    \hline
      $ \frac{1}{128}$  & $3.1146e-05$ & $5.9345e-08$ & $3.5599e-07$  \\   
    \hline
         $\frac{1}{256}$  & $3.7461e-05$ &  $8.1869e-08$ & $3.3017e-07$  \\
          \hline
         \end{tabular}
    \caption{Test (a3). Error of $\psi$ and mean and standard deviation of $j$ with respect to different $\e$.  }
   \label{T4}
   \end{center} 
  \end{table}

In Figure \ref{I-a1}, we see that solutions of the mean and standard deviation of $\rho$ and $j$ obtained from the GWPT match well with 
the reference solutions calculated from the time splitting spectral method \cite{BaoJin}. 
In the last row of Figure \ref{I-a1}, we compare the mean of $\rho$ and $j$ obtained from the GWPT for several values of $\e$ (lines) with the mean density and current obtained by classical mechanics (dots). 
One observes a tendency of convergence of solutions by using the GWPT to the classical case as $\e$ becomes smaller from $\e=\frac{1}{32}$ and $\e=\frac{1}{128}$ to $\e=\frac{1}{512}$.  
However, it is yet to be specified the precise description of the weak convergence, which remains a difficult question, as discussed at the end of subsection \ref{Semi-Limit}. 
{\color{black}In Table \ref{T0}, we compare the CPU time of using the GWPT and the time-splitting spectral method \cite{BaoJin} for Test (a1-ii) at output time $T=0.3$, with various $\e$ values. The experiment was done using MATLAB R2018b on macOS Mojave system with 2.4 GHz Intel Core i5 processor and 8GB DDR3 memory.
One can observe that the computational saving of the GWPT becomes more apparent as 
$\e$ decreases. The efficiency of the GWPT compared to the commonly used time-splitting spectral method is clearly demonstrated. 
}

In Tables \ref{T1}, \ref{T3} and \ref{T4}, one observes a uniform accuracy for $\psi$ and $j$ with respect to different small values of $\e$. 
Thus we conclude that in order to capture $\psi$ and $j$, 
$\e$-independent $N_{z,2}$ ($N_{z,2}=32$) can be used for all small $\e$, 
by putting $N_{z,1}$, $N_{z,3}$ and $N_{z,4}$ sufficiently large. 
In Table \ref{T2}, fixing $\e$, we see a fast spectral convergence of relative errors between using $N_{z,2}$ and $2N_{z,2}$ for $\psi$ and $j$, 
which indicates again that small $\e$-independent $N_{z,2}$ can be chosen to obtain accurate values of $\psi$ and $j$. 

%------------------------------------------------------------------------------
In the following Test (b) and Test (c), we let the initial data for $\psi$ depend on the random variable $z$ that 
follows a uniform distribution on $[-1,1]$. Let $N_{z,1}=N_{z,3}=N_{z,4}=500$, $N_{z,2}=32$ in Test (b) and Test (c). 
\\[2pt]

{\bf Test (b)} Let 
$$ V(x) = 1- \cos(x),  \qquad \psi(x,0, z)=A \exp\bigl[(i/\e)\bigl(\alpha_0 (x-q_0)^2 + p_0(x-q_0)\bigr)\bigr]. $$ 
Here we assume $q_0$ random, 
$$ q_0=\frac{\pi}{2}(1+0.5z), \qquad p_0=0, \qquad \alpha_0=i. $$

\begin{table}[H]
\begin{center}
\begin{tabular}{ |p{1cm} | p{2cm} | p{2cm}| p{2cm}| p{2cm}| p{2cm}| }
 \hline
  $\varepsilon$  &   $\text{Er}[\psi]$    &  $\text{Er}_1[j]$   &  $\text{Er}_2[j]$      \\
 \hline
    $\frac{1}{32}$ & $1.1114e-04$ &  $2.5535e-07$ & $1.5763e-06$  \\
  \hline
    $\frac{1}{64}$ &  $3.1764e-05$ &  $7.8356e-08$ & $2.4191e-06$   \\
     \hline
     $\frac{1}{256}$  &  $5.8151e-05$ & $1.5883e-07$ & $1.9183e-06$    \\
      \hline
    \end{tabular}
    \caption{Test (b). Error of $\psi$ and mean and standard deviation of $j$ with respect to different $\varepsilon$. $T=1$. }
\label{I-b}
 \end{center}
\end{table}

%------------------------------------------------------------------------------
{\bf Test (c)} Let
$$ V(x) = 1- \cos(x),  \qquad \psi(x,0, z)=A \exp\bigl[(i/\e)\bigl(\alpha_0 (x-q_0)^2 + p_0(x-q_0)\bigr)\bigr], $$ 
with $q_0$, $p_0$ depend on $z$, 
$$q_0=\frac{\pi}{2}(1+0.5z), \qquad p_0=0.5z, \qquad \alpha_0=i. $$ 

\begin{table}[H]
\begin{center}
\begin{tabular}{ |p{1cm} | p{2cm} | p{2cm}| p{2cm}| p{2cm}| p{2cm}| }
 \hline
  $\varepsilon$  &  $\text{Er}[\psi]$  &  $\text{Er}_1[j]$   &  $\text{Er}_2[j]$      \\
 \hline
   $\frac{1}{32}$  & $1.6159e-04$ &   $1.8413e-07$ & $5.9809e-07$  \\
  \hline
    $\frac{1}{64}$ &  $2.5781e-05$ &  $1.1032e-07$ & $1.9014e-07$  \\ 
  \hline
      $\frac{1}{256}$  & $2.5840e-05$ &  $1.1882e-08$ &  $1.3267e-07$  \\
   \hline
    \end{tabular}
    \caption{Test (c). Error of $\psi$ and mean and standard deviation of $\rho$ and $j$ with respect to different $\varepsilon$. $T=0.5$. }
\label{I-c}
 \end{center}
\end{table}
Table \ref{I-b} and Table \ref{I-c} give the same conclusion as Tests (a1)--(a3), that is, 
$\e$-independent $N_{z,2}$ (and sufficient large $N_{z,1}$, $N_{z,3}$, $N_{z,4}$) can be chosen to get accurate $\psi$ and $j$. 

Now we perform a two-dimensional random variable test: 
\\[2pt]

{\bf Test (d)}
Let the random potential be
\begin{equation}\label{2D-V}
V(x,z)=(1-\cos(x))(1+0.2 z_1+0.7 z_2), \end{equation}
where $z_1$, $z_2$ both follow uniform distributions on $[-1,1]$. The initial data of $\psi$ is given by (\ref{Psi-IC}). 
$\varepsilon=0.1$. Use $N_{z,1}=N_{z,2}=N_{z,3}=N_{z,4}=32$. 
Solutions and errors at time $T=1$ are shown in Figure \ref{2D} below. 

\begin{figure}[H]
\centering
\includegraphics[width=0.45\linewidth]{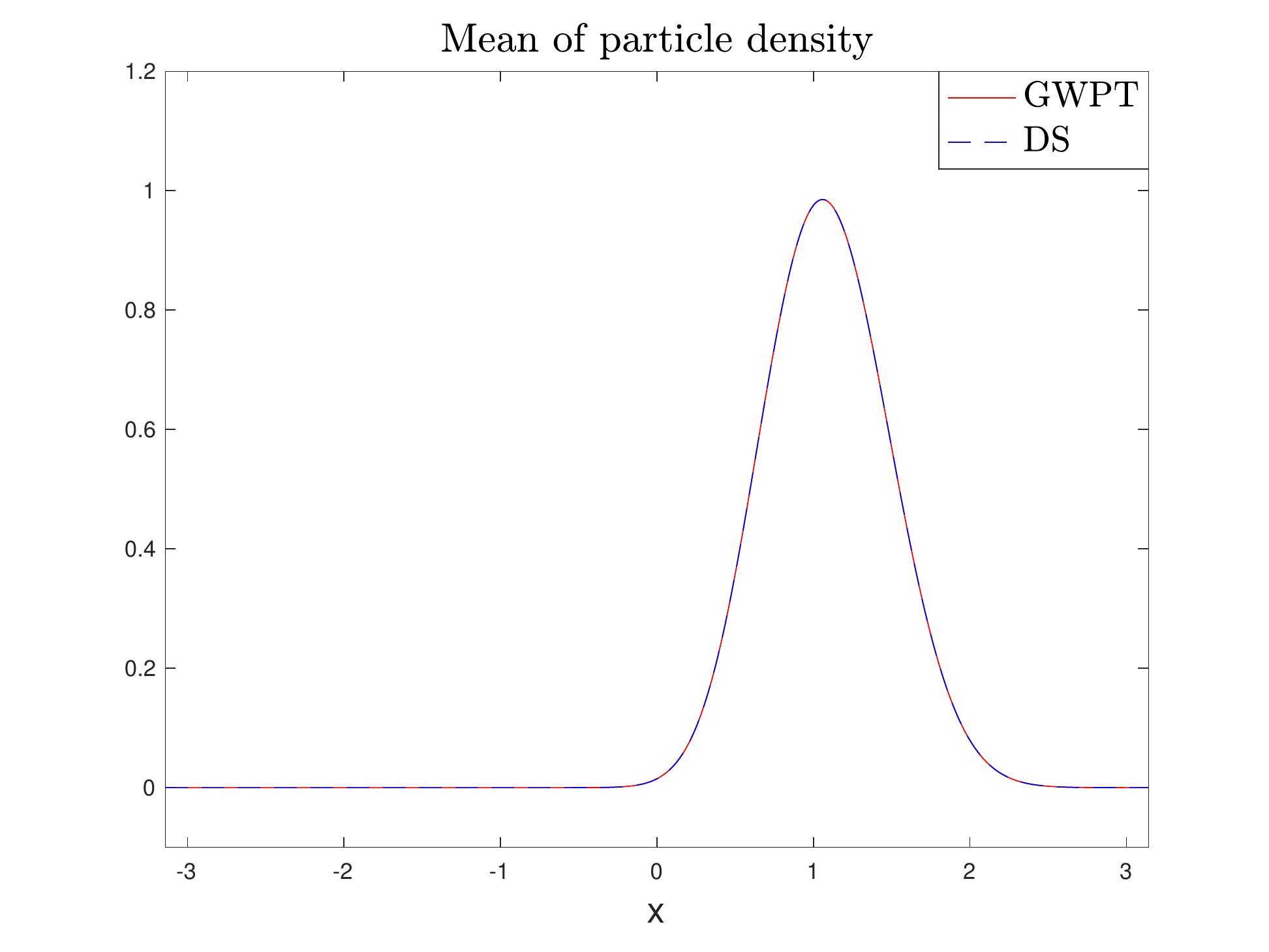}
\centering
\includegraphics[width=0.45\linewidth]{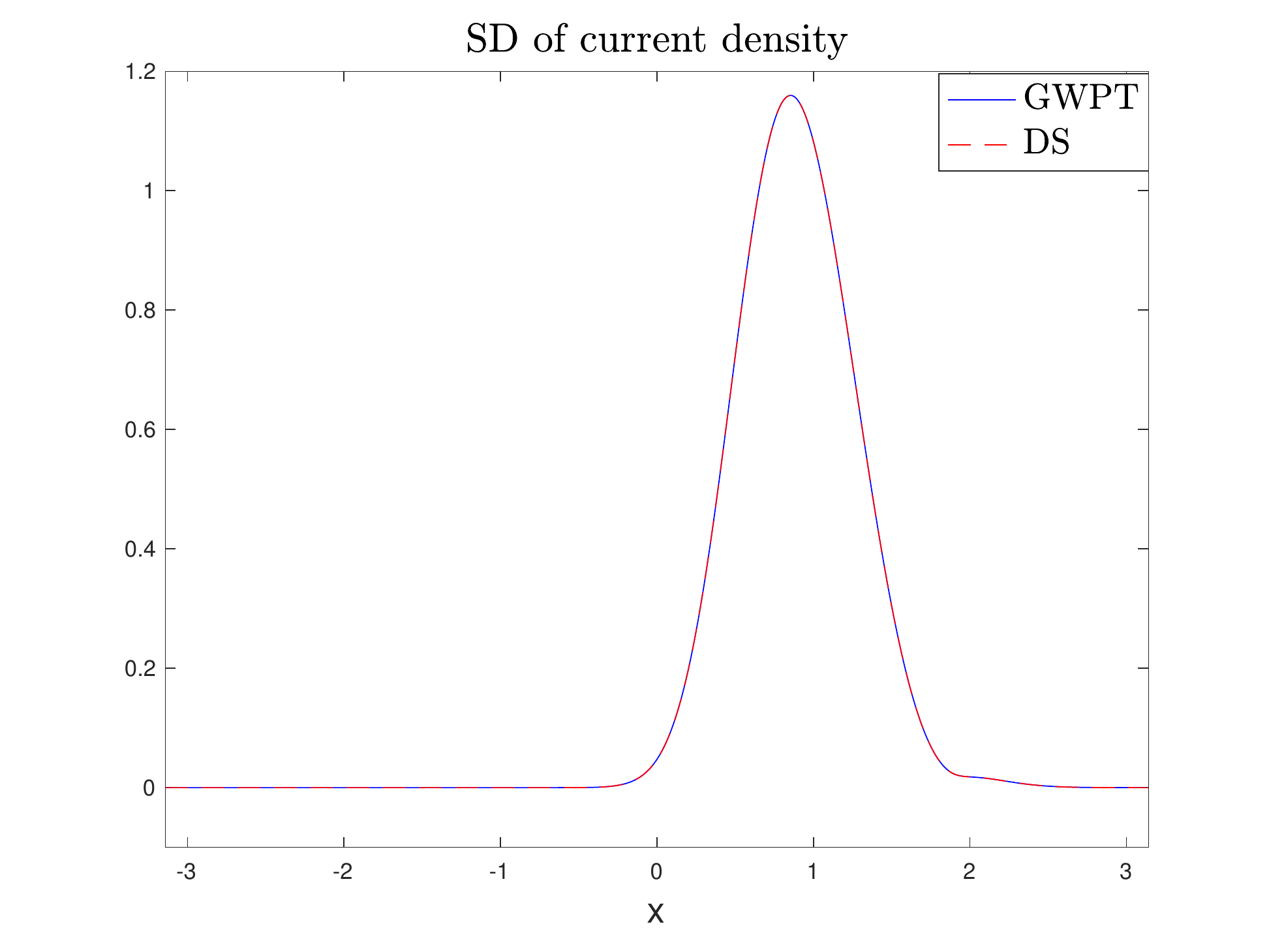}
\centering
\includegraphics[width=0.45\linewidth]{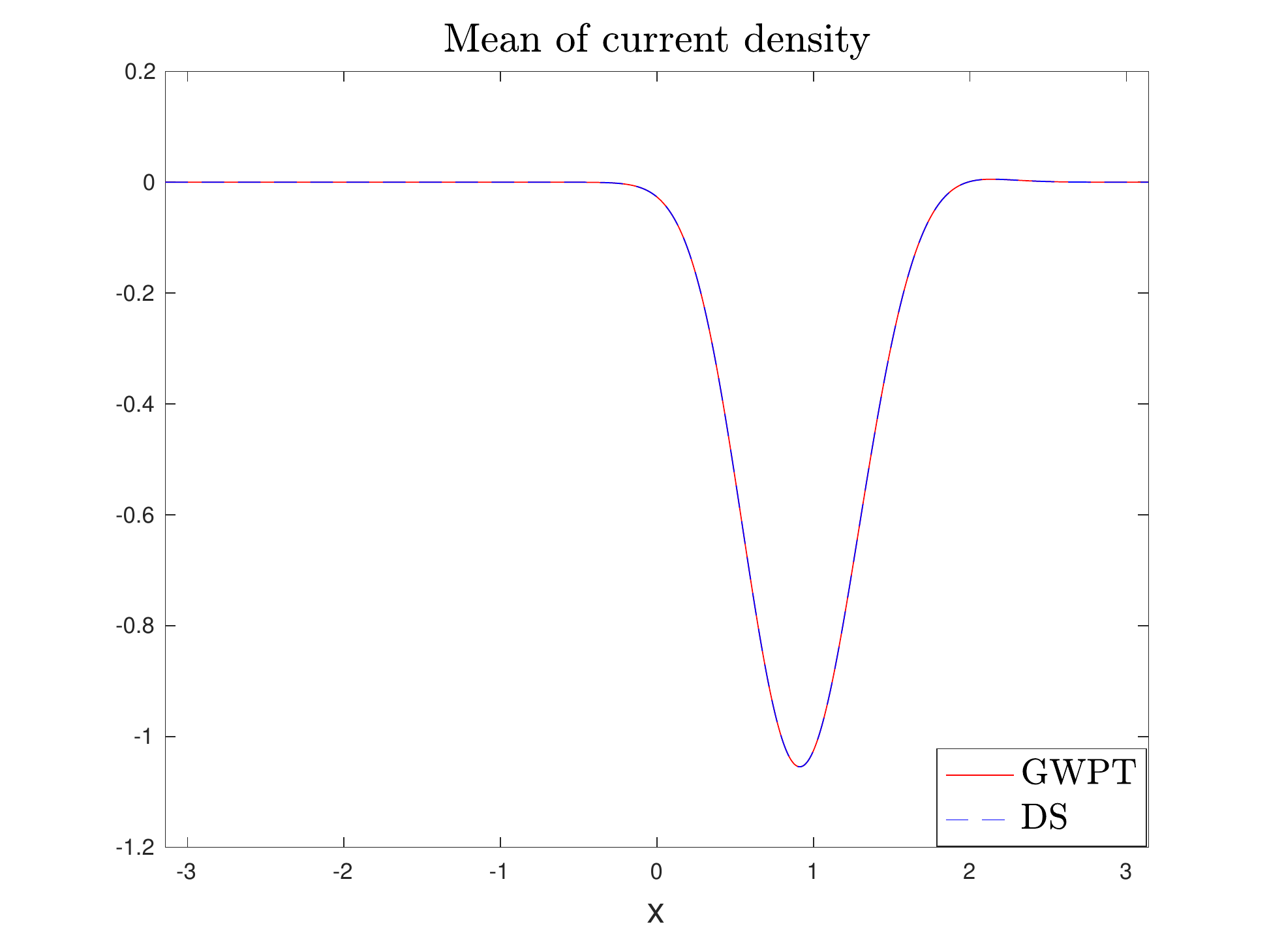}
\centering
\includegraphics[width=0.45\linewidth]{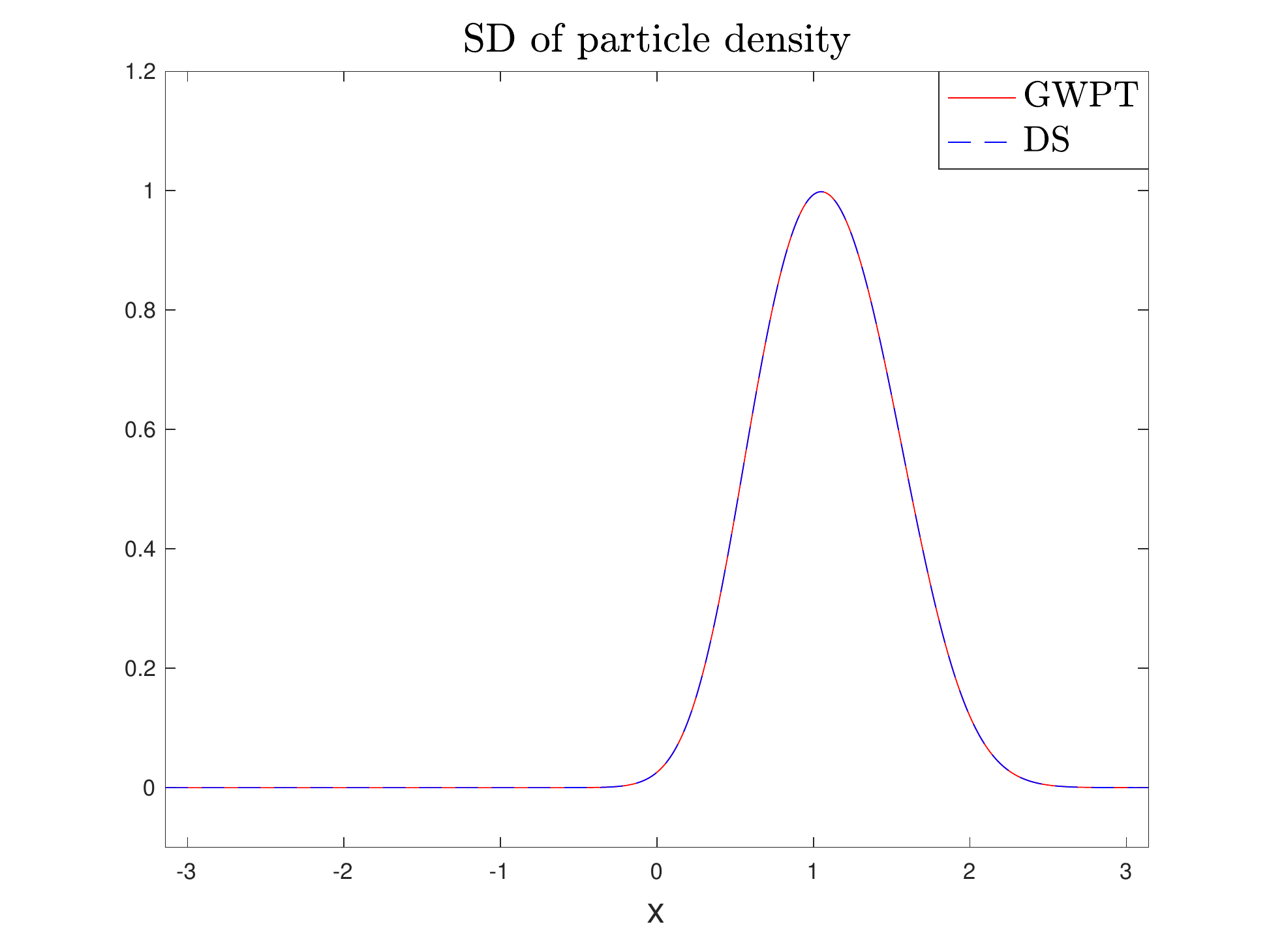}
\begin{table}[H]
\begin{center}
\begin{tabular}{ | p{2cm} | p{2cm} | p{2cm} | p{2cm} |p{2cm} |}
 \hline
 $\text{Er}[\psi]$   &  $\text{Er}_1[j]$   &  $\text{Er}_2[j]$    \\
\hline
$4.2501e-04$ &  $1.2864e-06$ & $1.7532e-06$  \\
\hline
\end{tabular}
 \end{center}
\end{table}
\caption{Test (d). Comparison of GWPT with DS. }
\label{2D}
\end{figure}
%% save in 'TestI-a1-2D.mat'

Figure \ref{2D} shows that we can capture accurately the mean and standard deviation of $\rho$ and $j$ when the random potential 
has a two-dimensional random variable. 
Since $\e$ in this test is not so small, $N_{z,1}$, $N_{z,2}$, $N_{z,3}$ and $N_{z,4}$ do not have to be very large. 
\\[2pt]

%------------------------------------------------------------------------------
\noindent {\bf Part II: Comparison between different perturbations in $V(x,z)$ and an error vs. time plot }

We compare three different orders of perturbations, 
using $\e$-independent $N_{z,1}$, $N_{z,2}$, $N_{z,3}$ and 
$\e$-dependent $N_{z,4}$. 

We first introduce Test (a4) and compare it with Test (a2) and Test (a3). 
\\[2pt]

{\bf Test (a4)}
We assume that $z$ follows the normal Gaussian distribution, and the random potential given by 
\begin{equation}\label{V2} V(x, z) = (1-\cos(x))(1+\sqrt{\e}z). \end{equation}

If one only needs the information of macroscopic quantities such as the 
current density, whose errors are measured by $\text{Er}_1[j]$ and $\text{Er}_2[j]$ in  (\ref{Er_j}), 
then the collocation points $N_{z, 1}$, $N_{z, 2}$, $N_{z, 3}$ 
used in the GWPT method can be chosen independently of small $\varepsilon$. 

From Tables \ref{P1}, \ref{P2} and \ref{P3} below, in which
$N_{z,1}=N_{z,2}=N_{z,3}=32$, $N_{z,4}=500$ and $T=1$. 
Recall that Test (a2) has a $(1+0.9z)$ perturbation in $V(x,z)$; 
Test (a3) has a $(1+\e z)$ perturbation in $V(x,z)$, and Test (a4) has a $(1+\sqrt{\e}z)$ perturbation in $V(x,z)$. 
One observes that errors for the mean and standard deviation of $j$ in the three tests are uniformly small with respect to 
different values of small $\e$. This is usually not the case for $\psi$, whose errors $\text{Er}[\psi]$ increase for smaller values of $\varepsilon$, 
because of its increasingly oscillatory behavior. 
However, the errors of $\psi$ in Table \ref{P1} are shown to be much smaller than that in Tables \ref{P2} and \ref{P3}, 
which indicates that if the random perturbation of the potential is relatively small, e.g., of $O(\e)$ perturbation in the test of Table \ref{P1}
compared to $O(\sqrt{\e})$ and $O(1)$ perturbation in the tests of Table \ref{P2} and \ref{P3} respectively, 
then $\e$-independent $N_{z,1}$, $N_{z,2}$, $N_{z,3}$ can be used to capture $\psi$ accurately. 

\begin{table}[H]
\begin{center}
\begin{tabular}{ |p{1cm} | p{2cm} | p{2cm}| p{2cm}| p{2cm}| p{2cm} | }
 \hline
  $\varepsilon$   & $\text{Er}[\psi]$  &  $\text{Er}_1[j]$   &  $\text{Er}_2[j]$      \\
 \hline
   $\frac{1}{32}$ & $2.8859e-05$ &   $9.5081e-08$ & $4.1857e-07$  \\
    \hline
     $\frac{1}{64}$ &  $2.9239e-05$ &  $8.4372e-09$ & $3.7980e-07$    \\ 
    \hline
      $ \frac{1}{128}$  & $3.1145e-05$ &  $5.9345e-08$ &  $3.5599e-07$  \\   
    \hline
         $\frac{1}{256}$  & $3.7461e-05$ &  $8.1869e-08$ & $3.3013e-07$  \\
          \hline
         \end{tabular}
    \caption{Test (a3). Error of $\psi$ and mean and standard deviation of $j$ with respect to different $\e$. }
  \label{P1}
   \end{center}
\end{table}

\begin{table}[H]
\begin{center}
\begin{tabular}{ |p{1cm} | p{2cm} | p{2cm}| p{2cm}| p{2cm}| p{2cm} | }
 \hline
  $\varepsilon$   & $\text{Er}[\psi]$  &  $\text{Er}_1[j]$   &  $\text{Er}_2[j]$      \\
 \hline
   $\frac{1}{32}$ &  $5.0359e-04$ &  $3.2411e-08$ & $5.8941e-07$ \\
   \hline
     $\frac{1}{64}$ &  $0.0020$ &  $6.2931e-08$ & $5.3575e-07$  \\ 
    \hline
      $ \frac{1}{128}$  &  $0.0093$ &  $1.0983e-07$ & $5.0308e-07$  \\   
    \hline
         $\frac{1}{256}$  &  $0.0535$ & $1.2743e-07$ & $4.5854e-07$  \\
          \hline
    \end{tabular}
    \caption{Test (a4). Error of $\psi$ and mean and standard deviation of $j$ with respect to different $\varepsilon$.}
    \label{P2}
   \end{center}
\end{table}

\begin{table}[H]
\begin{center}
\begin{tabular}{ |p{1cm} | p{2cm} | p{2cm}| p{2cm}| p{2cm}| p{2cm} | }
 \hline
  $\varepsilon$   &  $\text{Er}[\psi]$   &  $\text{Er}_1[j]$   &  $\text{Er}_2[j]$      \\
 \hline
   $\frac{1}{32}$ & $0.5842$ & $4.2896e-07$ & $1.8105e-05$   \\
    \hline
     $\frac{1}{64}$ &  $1.9324$ & $4.6665e-07$ & $8.8690e-06$  \\ 
    \hline
      $ \frac{1}{128}$  & $1.8047$ &  $4.7258e-07$ & $6.4676e-06$   \\   
    \hline
         $\frac{1}{256}$  & $1.6294$ & $4.1936e-07$ & $4.6450e-06$   \\
          \hline
         \end{tabular}
    \caption{Test (a2). Error of $\psi$ and mean and standard deviation of $j$ with respect to different $\e$. }
  \label{P3}
 \end{center}
\end{table}

%------------------------------------------------------------------------------
{\bf part (iv): } Semi-log error plot vs. time 

\begin{table}[H]
\begin{center}
\begin{tabular}{ |p{1cm} | p{2cm} | p{2cm}| p{2cm}| p{2cm}| p{2cm}|  }
 \hline
$T$  &  $\text{Er}[\psi]$ &  $\text{Er}_1[j]$   &  $\text{Er}_2[j]$      \\
 \hline
   $1$ & $3.9592e-05$ &  $1.9887e-07$ & $7.6387e-07$    \\ 
    \hline
     $2^{1/2}$  & $6.4139e-05$  & $1.1964e-06$ & $3.5285e-06$  \\
  \hline
     $2$ & $6.3926e-05$  & $6.5872e-07$ & $3.9316e-06$    \\ 
  \hline
      $2^{3/2}$ & $4.2088e-04$ & $3.9568e-06$ & $4.9694e-06$ \\
   \hline 
       $4$ &  $9.0106e-05$ &  $3.1581e-06$ & $2.6140e-06$      \\
    \hline 
       $2^{5/2}$ & $8.2447e-04$ &  $1.6763e-06$ & $3.6021e-05$  \\
    \hline 
       $8$ &  $0.0037$ & $1.1940e-05$ & $3.0379e-04$    \\
    \hline
      \end{tabular}
    \caption{Test (a4). Error of $\psi$ and mean and standard deviation of $j$ with respect to different time. $\e=1/128$. }
\label{Time-T}
\end{center}
\end{table}

\begin{figure}[H]
\centering
\includegraphics[width=0.6\linewidth]{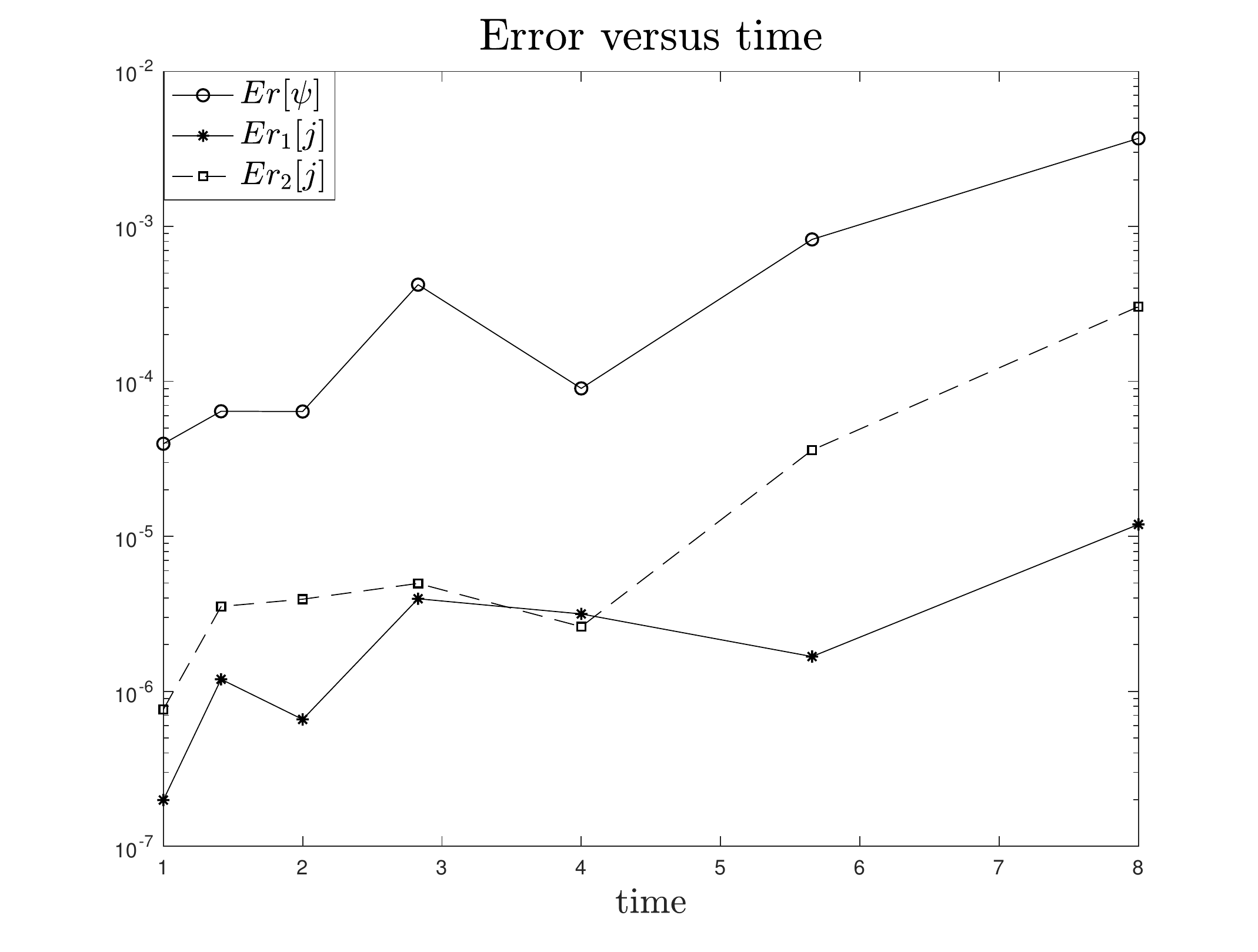}            
\caption{A semi-log plot of the errors versus time shown in Table \ref{Time-T}. }
\label{Time-P}
\end{figure}
%% save('Semilog.mat','T','y1','y2','y3');

In Table \ref{Time-T}, we compute the errors of 
$\psi$, and mean and standard deviation of $j$ with respect to different output time, 
with a corresponding semi-log plot shown in Figure \ref{Time-P}. 
One can observe that the overall trend of all the errors increase as time becomes longer. This trend is similar as that in the counterpart
deterministic problem \cite{GWPT}. 

%--------------------------------------------------------------
\section{Conclusion and future work}
\label{sec:5}

In this paper, we consider random potential or initial data in the semi-classical Schr{\"o}dinger equation. Based on the Gaussian wave packet transform 
numerical method studied for the deterministic problem in \cite{GWPT}, we adopt the stochastic collocation method to numerically compute the Schr{\"o}dinger equation with random inputs. We analyze how the number of collocation points needed at each step depends on the Planck constant 
$\e$, according to regularity analyses of the solution $\psi$ and $w$ in the random space. A variety of numerical experiments demonstrate our arguments and efficiency of the proposed numerical method. 

In the future, we propose to work on high-dimensional random variable test and develop efficient numerical solvers. 
Multidimensional physical space was considered in \cite{GWPT2}, and we expect to study the counterpart problem with uncertainties. 
In \cite{ZR}, the semi-classical Schr{\"o}dinger equation in the presence of electromagnetic field was reformulated by 
the Gaussian wave packets transform. It is interesting to consider random electromagnetic field and extend the numerical method studied
in \cite{ZR} to the random case.

\bibliographystyle{siam}
\bibliography{GWPT_Ref}
\end{document}